\documentclass{amsart}
\usepackage{mathtools}
\usepackage{xcolor}
\usepackage{tikz-cd}
\usepackage{leftidx}
\usepackage{lipsum}  
\usepackage{amssymb}

\newcommand{\lsp}{\operatorname{span}}
\newcommand{\clsp}{\overline{\lsp}}
\newcommand{\Toeplitz}[1]{\text{$\mathcal{T}C ^\ast(#1)$}}
\newcommand{\Kk}{\mathcal{K}}

\newcommand{\supp}{\operatorname{supp} }
\newcommand{\struc}{\operatorname{Struc}}
\newcommand\inner[2]{\left\langle #1, #2 \right\rangle}
\newtheorem{theorem}{Theorem}[section]
\newtheorem{lemma}[theorem]{Lemma}
\newtheorem{proposition}[theorem]{Proposition}
\newtheorem{corollary}[theorem]{Corollary}

\theoremstyle{definition}
\newtheorem{definition}[theorem]{Definition}

\theoremstyle{remark}
\newtheorem{remark}[theorem]{Remark}

\numberwithin{equation}{section}

\begin{document}

\title[Reconstruction of topological graphs and bimodules]{Reconstruction of topological graphs and their Hilbert bimodules}




\author{Rodrigo Frausino}
\address{R. Frausino, A.C.S. Ng, A. Sims\\ School of Mathematics and Applied Statistics\\ University of Wollongong\\ NSW 2522 AUSTRALIA}
\email{rsf934@uowmail.edu.au}

\author{Abraham C.S. Ng}
\email{abrahamn@uow.edu.au}

\author{Aidan Sims}
\email{asims@uow.edu.au}


\date{\today}

\thanks{This research was supported by Australian Research Council grants DP180100595 and DP200100155. We thank Adam Dor-On and the anonymous referee for drawing our attention to the work of Davidson--Katsoulis and Davidson--Roydor, and we thank Ken Davidson and Elias Katsoulis for helpful discussions of related literature.}



\begin{abstract}
We show that the Hilbert bimodule associated to a compact topological graph can be recovered from the $C^*$-algebraic triple consisting of the Toeplitz algebra of the graph, its gauge action and the commutative subalgebra of functions on the vertex space of the graph. We discuss connections with work of Davidson--Katsoulis and of Davidson--Roydor on local conjugacy of topological graphs and isomorphism of their tensor algebras. In particular, we give a direct proof that a compact topological graph can be recovered up to local conjugacy from its Hilbert bimodule, and present an example of nonisomorphic locally conjugate compact topological graphs with isomorphic Hilbert bimodules. We also give an elementary proof that for compact topological graphs with totally disconnected vertex space the notions of local conjugacy, Hilbert bimodule isomorphism, isomorphism of $C^*$-algebraic triples, and isomorphism all coincide.
\end{abstract}

\maketitle
\section{Introduction}
Given a class of mathematical objects, if for each object $X$ there is an associated structure $\struc(X)$, what kind of relationship does an isomorphism $\struc(X)\cong \struc(Y)$ impose between $X$ and $Y$? The Gelfand--Naimark Theorem gives us a classical example of this question: For $X$ and $Y$ locally compact Hausdorff spaces, an isomorphism $C_0(X)\cong C_0(Y)$ induces an isomorphism $X\cong Y$. Another similar result of Gelfand and Kolmogorov \cite{gelfand1939rings} deals with algebra homomorphisms and real-valued functions. A compilation of results of this kind can be found in \cite{garrido2002variations}. 

This question has recently been of significant interest in the context of graph algebras due to the work of Eilers and his collaborators (see for example \cite{dor2020classification}). In particular, in \cite{sims2020reconstructing}, \cite{BruceTak} and \cite{dor2020classification}, it was shown that a directed graph $E$ can be completely recovered from its Toeplitz algebra, its canonical gauge action, and its abelian coefficient subalgebra. In \cite{sims2020reconstructing} this was achieved for finite graphs using KMS theory; in \cite{BruceTak} this was extended to arbitrary discrete graphs using ground states, while in \cite{dor2020classification} they used nonselfadjoint operator-algebra theory.

The main purpose of this paper is to generalize the results of \cite{sims2020reconstructing} to totally disconnected compact topological graphs, but our approach yields interesting results, and questions, for more general compact topological graphs along the way. Drawing inspiration from \cite{sims2020reconstructing}, given a compact topological graph $E$, we use the KMS structure on its Toeplitz algebra, together with its coefficient subalgebra, to recover the corresponding graph bimodule. At first sight it may seem that our result can be recovered from \cite[Proposition 4.6 or Corollary 4.7]{dor2020classification}, but their results require that the coefficient algebra $A$ be a subalgebra of compact operators in a Hilbert space and hence must be of the form $\bigoplus_{i\in I}\Kk(H_i)$ for Hilbert spaces $(H_i)_{i\in I}$ \cite[Theorem 1.4.5]{arveson2012invitation}. If $A$ is also commutative, then each $H_i$ must be 1-dimensional and if it is unital as well, then $|I|<\infty$. Hence the intersection between our hypothesis and those of \cite{dor2020classification} yields the class of finite discrete graphs as in \cite{sims2020reconstructing}. 

In the final two sections of the paper, we consider the extent to which a topological graph $E$ can be recovered from its bimodule, and hence from the triple consisting of its Toeplitz $C^*$-algebra, gauge action and coefficient algebra. After our paper appeared on the arXiv, we discovered (we thank both Adam Dor-On and the anonymous referee for drawing our attention to the fact) that these questions were considered earlier by Davidson--Katsoulis \cite{DavKat} for local homeomorphisms, and by Davidson--Roydor \cite{DavRoy} for topological graphs. We present alternative proofs and examples for some of these results, but we emphasise that the definition of local conjugacy and our results relating isomorphism of Hilbert modules of topological graphs to local conjugacy of the graphs themselves both go back to the work of Davidson--Roydor, and before that to the work of Davidson--Katsoulis. Indeed, there is a substantial body of work on the relationship between isomorphism of graphs and of associated tensor algebras \cite{KatKri04, MuhSol98, MuhSol99, MuhSol00, Sol04}, and on the relationship between isomorphism, or local conjugacy, of multivariable dynamical systems and isomorphism of the associated Hilbert bimodules \cite{KakKat12, KakKat14, Kat18, KatRam22}.

In the preliminaries we recall some details about our main tool for studying the Toeplitz algebra of a compact topological graph --- the KMS-states for its gauge action. We also introduce Hilbert modules and the relationship between Hilbert bimodules over commutative $C^*$-algebras and Hilbert bundles.

In Section~\ref{section.A reconstruction result} we prove our first main theorem, Theorem~\ref{theorem.isoKMScompacttopgraph}: let $E$ and $F$ be compact topological graphs, $\gamma^E$ and $\gamma^F$ the gauge actions on the corresponding Toeplitz algebras $\Toeplitz{E}$ and $\Toeplitz{F}$, and $M_E$ and $M_F$ the coefficient algebras. We say that the triples $(\Toeplitz{E},\gamma^E,M_E)$ and $(\Toeplitz{F},\gamma^F,M_F)$ are isomorphic if there is a $\ast$-isomomorphism $\theta:\Toeplitz{E}\to \Toeplitz{F}$ that intertwines $\gamma^E$ and $\gamma^F$ and $\theta(M_E) = M_F$. Our theorem says that $(\Toeplitz{E},\gamma^E,M_E)$ and $(\Toeplitz{F},\gamma^F,M_F)$ are isomorphic if and only if the underlying graph bimodules $X(E)$ and $X(F)$ are isomorphic. We do this by proving that we can explicitly reconstruct $X(E)$ from $(\Toeplitz{E},\gamma^E,M_E)$, and make explicit the sense in which an isomorphism of Toeplitz triples induces an isomorphism of bimodules.

In Section~\ref{section.local} we show that we can do more than just reconstruct $X(E)$. We recall the notion of local conjugacy of topological graphs (\cite[Definition~4.3]{DavRoy}). We deduce that, for compact topological graphs, isomorphism of graph bimodules implies local conjugacy of graphs (this can be deduced from \cite[Theorem~4.5]{DavRoy}, but we give a direct proof). Next, we prove  Corollary~\ref{corollary.isomorphismoftotallydisconnectedtopgraph}: for compact topological graphs with totally disconnected vertex spaces, $(\Toeplitz{E},\gamma^E,M_E)$ and $(\Toeplitz{F},\gamma^F,M_F)$ are isomorphic if and only if $E$ and $F$ are isomorphic topological graphs (again, this follows from \cite[Theorem~5.5]{DavRoy}, but we give an elementary proof for zero-dimensional graphs).

In Section~\ref{section.example} we give an example that shows that in general, we cannot recover a compact topological graph from $(\Toeplitz{E},\gamma^E,M_E)$: we exhibit nonisomorphic topological graphs with isomorphic graph correspondences. Examples of nonisomorphic but locally conjugate topological graphs whose vertex spaces have convering dimension~1 also appear in \cite[Example~3.18]{DavKat}, and it follows from \cite[Theorem~5.5]{DavRoy} that their Hilbert modules are also isomorphic. However, our example is explicit and we describe a concrete isomorphism of the resulting Hilbert modules.

We finish off in Section~\ref{sec.epilogue} by giving a characterisation, in terms of cohomological data, of the vector-bundle structure associated to the right-Hilbert modules of compact topological graphs. Specifically, we demonstrate that they are precisely the vector bundles admitting local trivialisations whose transition functions take values in the permutation matrices. This is closely related to the characterisation of Kaliszewski, Patani and Quigg \cite{KaliPataQuigg} in terms of continuous choices of orthonormal basis. We indicate in a closing remark how this relates to the question of which pairs of locally conjugate topological graphs have isomorphic Hilbert modules.

\section{Preliminaries}
\subsection{KMS-states}
We first recall the Kubo--Martin--Schwinger states (KMS-states for short) of a $C^\ast$-dynamical system . See \cite{bratteli1981operatorvol1,bratteli1981operatorvol2} for details. By a $C^*$-dynamical system, we mean a strongly continuous action $\tau$ of $\mathbb{R}$ on a $C^\ast$-algebra $A$ by automorphisms; we call $\tau$ a \emph{dynamics} on $A$.

Let $\tau$ be a dynamics on a $C^\ast$-algebra $A$. An element $a \in A$ is \textit{analytic} for $\tau$ if the function $t \mapsto \tau_t(a)$ extends to an analytic function $z \mapsto \tau_z(a)$ from $\mathbb{C}$ to $A$. If it exists, this extension is unique.

Let $(A,\tau)$ be $C^\ast$-dynamical system, $\varphi$ a state on $A$ and $\beta \in \mathbb{R}$. We say $\varphi$ is a \textit{$\tau$-$\text{KMS}_\beta$-state} if
\begin{align*}
\varphi(a\tau_{i\beta}(b)) = \varphi(ba)
\end{align*}
for all analytic elements $a, b\in A$. When $\tau$ is implicit, we just say that $\varphi$ is a $\text{KMS}_\beta$-state. The set of KMS$_\beta$-states is convex and weak$^\ast$-compact (\cite[Theorem 5.3.30]{bratteli1981operatorvol2}). An extremal point of this set is called an \textit{extremal} KMS$_\beta$-state.

By a KMS$_\infty$ state, we will mean a weak$^*$-limit of a sequence $(\phi_n)_{n=1}^\infty$ of states such that each $\phi_n$ is a KMS$_{\beta_n}$-state for a sequence $(\beta_n)_{n=1}^\infty$ satisfying $\beta_n \to \infty$.

\subsection{Hilbert modules}
There are plenty of references that explore the properties of Hilbert modules, for example \cite{lance1995hilbert}. We rely mainly on \cite{raeburn1998morita} for a slightly more recent approach. Let $A$ be a $C^*$-algebra. A \emph{right Hilbert $A$-module} is a right $A$-module $V$ equipped
with a map $\langle \cdot, \cdot \rangle_A : V \times V \to A$, which is linear in the second variable and such that for $x,y \in V$, $a \in A$,
\begin{itemize}
    \item $\langle x,x \rangle_A \geq 0$, with equality only when $x = 0$;
    \item \label{A-linear in the second variable} $\langle x,y \cdot a\rangle_A=\langle x,y\rangle_Aa$;
    \item\label{it:adjoint of ip}$\langle x,y\rangle_A = \langle y,x\rangle_A^*$; and
    \item $X$ is complete in the norm defined
        by $\|x\|_A ^2 = \|\langle x, x\rangle_A\|$.
\end{itemize}

A map  $T:V \to V$  is an \emph{adjointable operator} if there exists $T^*:V \to V$ called the \emph{adjoint} such that $\langle
 T^*y,x\rangle_A=\langle y,Tx\rangle_A$, for all $x$, $y \in V$. The adjoint $T^*$ is unique, and
 $T$ is automatically a bounded, linear $A$--module homomorphism. The set $\mathcal{L}(V)$ of adjointable operators on $V$
 is a $C^\ast$-algebra.

A \emph{$C^\ast$-correspondence} over $A$ (or right--Hilbert $A$-bimodule) is a right Hilbert $A$-module $X$ together with a left action of $A$ by adjointable operators on $X$ which is given by a $\ast$-homomorphism $\phi\colon A\to \mathcal{L}(X)$, in the sense that the left action $a\cdot x$ is given by $\phi(a)x$. This implies the familiar-looking formula 
 \begin{equation}\label{eq:adjointable left action}
     \langle a \cdot y,x \rangle_A=\langle y,a^\ast\cdot x\rangle_A \quad\text{ for all $a \in A$, $x$, $y \in X$.}
 \end{equation}
 
 Every $C^*$-algebra $A$ can be regarded as a $C^*$-correspondence over itself with actions given by multiplication, and inner product $\langle a, b\rangle_A = a^*b$. We denote this $C^*$-correspondence by ${_A A_A}$.

The internal tensor product $X \otimes_A Y$ of $C^*$-correspondences over $A$ is the completion of the quotient of the algebraic tensor product $X\odot Y$ by the submodule generated by differences of the form  $a \odot y - x \odot a\cdot y$, with respect to the inner product satisfying $\langle x \odot y, x' \odot y'\rangle_A = \langle y, \langle x, x'\rangle_A \cdot y'\rangle_A$ (see \cite[Chapter~8]{raeburn2005graph}). For $n \ge 1$, we write
\[
X^{\otimes n} = \underbrace{X \otimes_A X \otimes_A \cdots \otimes_A X}_{\text{$n$ terms}}.
\]
By convention, $X^{\otimes 0} = {_A A_A}$.

Let $A$ be a $C^*$-algebra and let $X$ be a $C^*$-correspondence over $A$. A \emph{Toeplitz 
representation} of $X$ in a $C^*$-algebra $B$ is a pair $(\psi,\pi)$ such that $\psi:X \to {_B B_B}$ and $\pi:A \to B$ together constitute a Hilbert bimodule homomorphism. 

Given a Toeplitz representation $(\psi,\pi)$ of $X$ in $B$, for each $n\in \mathbb{N}$ there exists \cite[Proposition~1.8]{fowler1999toeplitz} a Toeplitz representation $(\psi^{\otimes n},\pi)$ of $X^{\otimes n}$ in $B$ such that for all $x_1,\ldots,x_n\in X$,
 \[\psi^{\otimes n}(x_1\otimes x_2\otimes...\otimes x_n)=\psi(x_1)\psi(x_2)...\psi(x_n).\]
 
 Given a $C^*$-correspondence $X$ over $A$, there exists a Toeplitz representation
 $(\iota_X,\iota_A)$ of $X$ in a $C^\ast$-algebra $\mathcal{T}_X$ that is \emph{universal} for Toeplitz representations \cite[Proposition 1.3]{fowler1999toeplitz}, \cite[Theorem 3.4]{pimsner12class}. This means that for another Toeplitz representation $(\psi,\pi)$ of $X$ in a $C^\ast$-algebra $B$, there exists a homomorphism denoted $\psi\times \pi:\mathcal{T}_X\to B$, such that $(\psi\times \pi)\circ \iota_X =\psi$ and $(\psi\times \pi)\circ \iota_A=\pi$. This $\mathcal{T}_X$ is called the \textit{Toeplitz algebra} of the correspondence $X$. By, \cite[Proposition 8.9]{raeburn2005graph},
 \[\mathcal{T}_X=\clsp\{\iota_X^{\otimes m}(\xi)\iota_X^{\otimes n}(\eta)^\ast\mid m,n\geq 0, \xi\in X^{\otimes m}, \eta\in X^{\otimes n}\}. \]

The $C^*$-algebra $\iota_A(A)$ (or  $A$) is called the \textit{coefficient algebra} of $\mathcal{T}_X$.
By \cite[Proposition 1.3]{fowler1999toeplitz} there exists a strongly continuous action $\gamma^X:\mathbb{T}\to \operatorname{Aut}(\mathcal{T}_X)$ such that for $z\in\mathbb{T}$, $\gamma^X_z(\iota_{A}(a))=\iota_{A}(a)$ and $\gamma^X_z(\iota_{X}(x))=z\iota_{X}(x)$ for $a\in A$ and $x\in X$, called the \textit{gauge action}. The dynamics that we will consider is the action $t\mapsto \gamma^X_{e^{it}}$.

    For $n \in \mathbb{Z}$, the $n^{\text{th}}$ spectral subspace $(\mathcal{T}_X)_n$ of $\mathcal{T}_X$ for the action $\gamma$ is the space $\{a \in \mathcal{T}_X : \gamma_z(a) = z^n a\text{ for all }z \in \mathbb{T}\}$. One can check that 
\[
(\mathcal{T}_X)_n = \clsp\{\iota_X^{\otimes p}(\xi)\iota_X^{\otimes q}(\eta)^\ast\mid p,q\geq 0, p-q = n, \xi\in X^{\otimes p}, \eta\in X^{\otimes q}\}.
\]
We will make frequent use of the first spectral subspace $(\mathcal{T}_X)_1$ later.

Let $A,B$ be $C^\ast$-algebras and $X,Y$ be $C^\ast$-correspondences over $A$ and $B$, respectively. Then, a map $\theta:\mathcal{T}_X\to \mathcal{T}_Y$ will be called an \textit{isomorphism of triples} if it is an isomorphism of $C^\ast$-algebras such that $\theta\circ \gamma^X_z=\gamma^Y_z\circ \theta$ for all $z\in \mathbb{T}$ and carries $\iota_A(A)$ onto $\iota_B(B)$.
\subsection{Topological graphs and modules}\label{subsec:topgraphmod}

 A \textit{topological graph} is a quadruple $E=(E^0,E^1,r,s)$ where the \textit{vertex set} $E^0$ and the \textit{edge set} $E^1$ are locally compact Hausdorff spaces, the \textit{range map} $r\colon E^1\to E^0$ is a continuous function and the \textit{source map} $s\colon E^1\to E^0$ is a local homeomorphism. A \textit{path} in $E$ is a finite sequence of edges $\mu = \mu_1\mu_2 \dots \mu_n$ such that $s(\mu_{i}) = r(\mu_{i+1})$ for $1\le i \le n-1$ or a vertex $v\in E^0$. For such a path, $|\mu| = n$ is called the \textit{length} of $\mu$ and by convention $|v| = 0$ for all $v\in E^0$. For $n\ge 1$, the set of paths of length $n$ is denoted $E^n$ and the vertex set $E^0$ is considered the set of paths of length $0$. Define $E^\ast := \bigcup_{n\in \mathbb{N}} E^n$ to be the set of all paths. For $\mu \in E^n$, let $r(\mu) := r(\mu_1)$ and $s(\mu):= s(\mu_n)$. For $v\in E^0$, we define $E^\ast v = \{\mu \in E^\ast : s(\mu) = v\}$ and $vE^\ast=\{\mu \in E^\ast : r(\mu) = v\}$, with analogous notation we can define $vE^n$ and $E^n v$. More generally, for a subset $U \subseteq E^0$ we write $E^1 U$ for $s^{-1}(U)$.
\begin{lemma}\label{lemma:s-sectionslocalconstant}
Let $E$ be a compact topological graph. Then the function $v \mapsto |E^1 v|$ is locally constant. For each $v\in E^0$, there exist an open neighborhood $W$ of $v$ and disjoint open s-sections $(Z_e)_{e\in E^1 v}$ such that $s(Z_e)=W$ for all $e\in E^1 v$ and $E^1 W = \bigsqcup_{e\in E^1 v} Z_e$.
\end{lemma}
\begin{proof}
Fix $v\in E^0$. Since $s$ is local homeomorphism, $E^1 v$ is a discrete subspace of the compact space $E^1$ and thus finite. For each $e\in E^1 v$ there exists a neighborhood $V_e$ of $e$ on which $s$ is a homeomorphism. Since $E^1$ is Hausdorff and $E^1v$ is finite, by shrinking the $V_e$, we may assume that $V_e$, $e\in E^1v$ are pairwise disjoint. Since $s$ is a local homeomorphism, it is an open map, and so $W_0 := \bigcap_{e\in E^1 v}s(V_e)$ is open. For $w \in W_0$ we have $|E^1 w| \ge |\bigcup_{e \in E^1 v} V_e \cap E^1 w| = |E^1 v|$. For the first statement, it suffices to show that there is an open subset $W \subseteq W_0$ containing $v$ such that $|E^1 w| \le |E^1v|$ for all $w \in W$. We suppose otherwise and derive a contradiction. Then there is a sequence $(w_i)$ in $W_0$ converging to $v$ such that $|E^1 w_i| > |E^1 v|$ for all $i$. It follows that for each $i$ there exists $e_i \in E^1 w_i \setminus \bigcup_{e \in E^1 v} V_e$. Since $E^1$ is compact, by passing to a subsequence, we may assume that the sequence $(e_i)_i$ is convergent, say to $e_\infty \in E^1$. Continuity of the source map forces $s(e_\infty) = \lim_i s(e_i) = \lim_i w_i = v$ so that $e_\infty \in E^1 v$. So by definition we have $e_\infty \in V_{e_\infty}$ but $e_i\notin V_{e_\infty}$ for all $i$, contradicting that $e_i\to e_\infty$.

For the second statement, for each $e\in E^1 v$ take $V_e$ and $W$ as before, and let $Z_e = V_e\cap E^1 W$. Then $s:Z_e\to W$ is a homeomorphism for each $e\in E^1 v$ and $E^1 W=\bigsqcup_{e\in E^1 v}Z_e$, since, otherwise there would be an edge $e'\in E^1W$ such that $e'\notin V_e$ for all $e\in E^1 v$. This would lead to $|E^1 s(e')|>|E^1 v|$ and $s(e')\in W$, a contradiction.
\end{proof}

\begin{remark}\label{rem:maxexist}
Lemma \ref{lemma:s-sectionslocalconstant} implies that the map $v\mapsto |E^1v|$ is continuous and in particular, $\max_{v\in E^0}|E^1v|$ exists and is finite as $E^0$ is compact. By induction, $\max_{v\in E^0}|E^nv|$ exists and is finite for all $n$.
\end{remark}

What follows comes from \cite{katsura2004class}. Let $E$ be a topological graph. There are a right action of
$C_0(E^0)$ on $C_c(E^1)$ and a $C_0(E^0)$-valued inner product on $C_c(E^1)$ such that
\[(x\cdot \alpha)(e)= x(e)\alpha(s(e))\quad\text{and}\quad \inner{x}{y}_{C_0(E^0)}(v)= \sum_{s(e)=v}\overline{x(e)}y(e)\]
for $x,y \in C_c(E^1)$ and $\alpha\in C_0(E^0)$. If $E^1 v=\emptyset$, our convention is that the sum is equal to $0$. The completion $X(E)$ of $C_c(E^1)$ in the norm $\|x\|^2 = \|\langle x,x\rangle\|_{C_0(E^0)}$ is a Hilbert $A$-module and is equal to
\begin{equation}\label{eq:defX(E)graphcorrespondence}
    \{x\in C(E^1): \inner{x}{x}\in C_0(E^0)\}.
\end{equation}
 The formula
\[(\alpha\cdot x)(e)\coloneqq \alpha(r(e))x(e)\]
defines an action of $C_0(E^0)$ by adjointable operators on $X(E)$, so that $X(E)$ becomes a $C^*$-correspondence over $C_0(E^0)$. We call this the \emph{graph correspondence} associated to $E$. The Toeplitz algebra of the topological graph $E$ is then $\Toeplitz{E}\coloneqq \mathcal{T}_{X(E)}$.

There is a categorical equivalence between Hilbert $C(X)$-modules, for $X$ a compact Hausdorff space and Hilbert bundles over $X$ \cite{dixmier1963champs,dupre1983banach,takahashi1979duality}. Following \cite{dupre1983banach}, by a Hilbert bundle over $X$ we mean a triple $(p,E,X)$ where $p:E\to X$ is an open surjection, $E$ and $X$ are topological spaces, together with operations and inner products making each fiber $E_x=p^{-1}(x)$ a Hilbert space that satisfies certain compatibility conditions \cite[Definition~13.4]{fell1988representations}. We describe briefly how to construct the \textit{canonical Hilbert bundle} $\mathcal{E}_V$ of a Hilbert $C(X)$-module $V$. For each $x\in X$, let
\[
I_x \coloneqq \{f\in C(X): f(x)=0\}\qquad\text{and}\qquad
J_x \coloneqq V I_x = \{v\cdot f: v\in V,\,f\in I_x \}.
\]
In the rest of this section we will drop the subscript on the inner product. By the Hewitt--Cohen factorization theorem  \cite[Proposition~2.31]{raeburn1998morita} $J_x$ is a closed submodule of $V$ and by \cite[Lemma~3.32]{raeburn1998morita} we can write $J_x = \{v\in V: \inner{v}{v}\in I_x\}$. Since $J_x$ is a closed submodule, we may take the quotient $V/J_x$, which is a Hilbert $C(X)/I_x$-module. As $C(X)/I_x \cong \mathbb{C}$, the quotient $V/J_x$ is a Hilbert space. Let $\pi_x:V\to V/J_x$ be the quotient map and let
\[\textstyle 
    E\coloneqq \bigsqcup_{x\in X}V/J_x.
\]
Define $p:E\to X$ by $p(v+J_x)=x$ for each $x\in X$. By \cite[Proposition~3.25]{raeburn1998morita}, for each $v\in V$ and $x\in X$ we have $\|\pi_x(v)\|^2=\inner{v}{v}(x)$ and hence, the map $x\mapsto \|\pi_x(v)\|$ is continuous. For $v \in V$, define $\hat{v} : X \to E$ by $\hat{v}(x) = \pi_x(v)$. Let $\Gamma := \{\hat{v} : v \in V\}$. Then $\Gamma$ is a complex linear space of cross-sections and for each $x\in X$, the set $\{f(x):f\in \Gamma\} = \{\pi_x(v) : v \in V\}$ is equal to $V/J_x$. Hence, by \cite[Theorem~II.13.18]{fell1988representations} there is a unique topology on $E$ making $\mathcal{E}_V=(p,E,X)$ into a Hilbert bundle and all elements of $\Gamma$ continuous cross-sections. The map $v \mapsto \hat{v}$ is an isomorphism between the Hilbert $C(X)$-modules $V$ and $\Gamma$. We will be interested only in compact topological graphs, so the construction above can be applied to the graph correspondences in this paper.

\section{Reconstruction of topological--graph bimodules}\label{section.A reconstruction result}

From now on, $E$ denotes a compact topological graph. That is, $E^0$ and $E^1$ are compact, so $r$ and $s$ are proper maps. We use results from Hawkins' thesis \cite{hawkins2015applications}, which we reproduce here for the convenience of the reader. In the following, $\mathcal{M}(X)$ denotes the space of finite Borel measures on a compact Hausdorff space $X$ and $\mathcal{M}^1(X) \subset \mathcal{M}(X)$ denotes the space of Borel probability measures on $X$. By Remark \ref{rem:maxexist}, we can also define \[\rho(A_E)\coloneqq \lim_{n\to \infty} \max_{v\in E^0}\vert E^n v\vert ^{1/n}<\infty,\] which we call the \textit{spectral radius} associated with the compact topological graph $E$ (the notation is motivated by the spectral radius of the adjacency matrix in the discrete case).

\begin{theorem}[{{\cite[Theorem 5.1.10]{hawkins2015applications}}}]\label{theorem.topgraphkmsepsilonstate}
Let $E$ be a compact topological graph, and fix $\beta>\log(\rho(A_E))$. For each $\epsilon\in \mathcal{M}(E^0)$ satisfying 
\begin{equation}\label{eq:normalising}
    \int_{E^0}\sum_{\mu\in E^\ast v}e^{-\beta |\mu|}\,d\epsilon(v)=1,
\end{equation}
there exists a KMS$_\beta$-state $\phi_\epsilon$ such that for $x\in X(E)^{\otimes k}$ and $y\in X(E)^{\otimes \ell}$
\begin{equation}
    \phi_\epsilon(\iota_{X(E)}^{\otimes k}(x)\iota_{X(E)}^{\otimes \ell}(y)^\ast)=\delta_{k,\ell}e^{-\beta k}\int_{E^0}\sum_{\mu\in E^\ast v}e^{-\beta|\mu|}\inner{y}{x}_{C(E^0)}(r(\mu))\,d\epsilon(v).
\end{equation}
\end{theorem}

For each $v \in E^0$ and each $\beta > \log(\rho(A_E))$, we write
\[
N^\beta_v := \sum_{\mu \in E^* v} e^{-\beta|\mu|}.
\]
For $v \in E^0$, the measure $\varepsilon_v := (N^\beta_v)^{-1} \delta_v$ satisfies~\eqref{eq:normalising}, so Theorem~\ref{theorem.topgraphkmsepsilonstate} supplies an associated KMS$_\beta$-state $\phi^\beta_v :=\phi^\beta_{\varepsilon_v}$. If $\beta$ is clear from context, we just write $\phi_v$ for $\phi^\beta_v$. The proof of Theorem~\ref{theorem.topgraphkmsepsilonstate} shows that for $a\in \Toeplitz{E}$,
\begin{equation}\label{eq.kmsstateinanelement}
    \phi^\beta_v(a)=\frac{\sum_{\mu\in E^\ast v}e^{-\beta|\mu|}\inner{\psi_v\times \pi_v(a)e_\mu}{e_\mu}}{N^\beta_v}.
\end{equation}

\begin{theorem}[{{\cite[Theorem 5.1.11]{hawkins2015applications}}}]\label{theorem.isoKMScompacttopgraph}
Let $E$ be a compact topological graph, and fix $\beta>\log(\rho(A_E))$.
Then there is an affine isomorphism of $\mathcal{M}^1(E^0)$ onto the set of KMS$_\beta$-states of \Toeplitz{E} that takes a measure $\Omega$ to the state $\varphi_\Omega$ given by
\[\varphi_\Omega (a)=\int_{E^0}\phi_v(a)\,d\Omega(v) \quad\text{for all $a \in \Toeplitz{E}$}.\]
\end{theorem}

\begin{remark}\label{rem.extremalKMSstates}
The proof of Theorem \ref{theorem.isoKMScompacttopgraph} shows, among other things, that $v\mapsto \phi_v$ is a homeomorphism of $E^0$ onto the space of extremal points of the set of KMS$_\beta$-states for $\tau$.
\end{remark}

Replicating the proof of \cite[Lemma 4.1.7]{hawkins2015applications} yields the following proposition bar the last statement which we will prove.

\begin{proposition}\label{proposition.representationfaithful} Let $E$ be a topological graph, fix $v \in E^0$ and let $\{e_\mu: \mu\in E^\ast v\}$ denote the canonical basis for $\ell^2(E^\ast v)$. Then there exists a linear map $\psi_v:X(E)\to \mathcal{B}(\ell^2(E^\ast v))$ and a homomorphism $\pi_v:C_0(E^0)\to \mathcal{B}(\ell^2(E^\ast v))$ such that for all $\xi\in X(E)$, $\alpha\in C_0(E^0)$ and $\mu\in E^\ast v$,
\begin{align*}
    \psi_v(\xi)e_\mu=\sum_{f\in E^1 r(\mu)}\xi(f)e_{f\mu}\quad \text{and}\quad  \pi_v(\alpha)e_\mu=\alpha(r(\mu))e_\mu.
\end{align*}
The pair $(\psi_v,\pi_v)$ is a Toeplitz representation of $\Toeplitz{E}$ on $\ell^2(E^\ast v)$ and the direct sum $\bigoplus_{v\in E^0} (\psi_v\times \pi_v)$ is faithful.\end{proposition}
\begin{proof} All but the statement that $\Theta := \bigoplus_{v\in E^0}(\psi_v\times \pi_v)$ is faithful follows from essentially replicating the computations in the proof of \cite[Lemma 4.1.7]{hawkins2015applications}, hence we will only deal with the question of faithfulness.  

Let $U_v$ be the canonical inclusion of $\ell^2(E^\ast v)$ into $\ell^2(E^\ast)$. This $U_v$ is an isometry with adjoint $U_v^\ast:\ell^2(E^\ast)\to\ell^2(E^\ast v)$ the corresponding orthogonal projection. We define $U:= \bigoplus_{v\in E^0} U_v :\bigoplus_{v\in E^0}\ell^2(E^\ast v)\to \ell^2(E^\ast)$. Then $U$ is unitary, and
\[U^\ast h= (h|_{E^\ast v})_{v\in E^0}\quad \text{for all }h\in \ell^2(E^\ast).\]
Let $\{\delta_\mu : \mu \in E^\ast\}$ be the canonical basis for $\ell^2(E^\ast)$. By \cite[Lemma 4.1.7]{hawkins2015applications}, there exist $\lambda^0: C_0(E^0) \to \mathcal{B}(\ell^2(E^\ast))$ and $\lambda^1 : X(E) \to \mathcal{B}(\ell^2(E^\ast))$ such that for all $\xi\in X(E), \alpha\in C_0(E^0)$ and $\mu\in E^\ast$,
\[\lambda^1(\xi)\delta_\mu=\sum_{f\in E^1 r(\mu)}\xi(f)\delta_{f\mu}\quad \text{and}\quad  \lambda^0(\alpha)\delta_\mu=\alpha(r(\mu))\delta_\mu.\] Furthermore, $(\lambda^1,\lambda^0)$ is a Toeplitz representation of $X(E)$ on $\ell^2(E^\ast)$ and $\lambda^1\times\lambda^0$ is faithful.
We claim that, for $a\in\Toeplitz{E}$, 
\begin{equation}\label{eq.unitaryequivalenceofproposition}
\Theta(a)=U^\ast\left(\lambda^0\times \lambda^1(a)\right)U.
\end{equation}
To show this, we first prove that $U_v (\psi_v \times \pi_v(a)) = (\lambda^1 \times\lambda^0(a)) U_v$ for all $v \in E^0$.
By linearity and continuity, it suffices to fix $x=x_1\otimes x_2\otimes ... \otimes x_m\in X(E)^{\otimes m}$, $y=y_1\otimes y_2\otimes...\otimes y_n\in X(E)^{\otimes n}$ (with the convention that if $m = 0$, we mean $x = b \in A$, and similarly if $m = 0$ for $y$), and then consider $a = \iota_{X(E)}^{\otimes m}(x)\iota_{X(E)}^{\otimes n}(y)^\ast$. With the convention that $\prod^0_{i=1} \psi_v(x_i)$ means $\pi_v(b)$ when $x = b \in A$, and similarly for $y$, we have

\[U_v \Big( \psi_v\times\pi_v(a)\Big)=U_v\Big(\prod_{i=1}^m\psi_v(x_i)\Big)\Big(\prod_{i=0}^{n-1}\psi_v(y_{n-i})^\ast\Big).\]
Direct calculation with basis vectors shows that $U_v\psi_v(x_i)=\lambda^1(x_i) U_v$ and that $U_v\psi_v(y_j)^\ast=\lambda^1(y_j)^\ast U_v$. Therefore
\begin{align*}
U_v \Big(\psi_v\times\pi_v(a)\Big)
    =\Big(\prod_{i=1}^m\lambda^{1}(x_i)\Big)\Big(\prod_{i=0}^{n-1}\lambda^1(y_{n-i})^\ast\Big) U_v
    = \big(\lambda^1\times\lambda^0(a)\big)U_v
\end{align*}
as claimed.

Since $U_v$ is an isometry, we deduce that $\psi_v\times\pi_v(a)=U_v^\ast(\lambda^1\times \lambda^0(a))U_v$. Since each $\ell^2(E^\ast v) \subseteq \ell^2(E^*)$ is invariant for $\lambda^1\times\lambda^0(a)$,
\begin{align*}
\Theta(a)&=\bigoplus_{v\in E^0}\Big(\psi_v\times \pi_v(a)\Big)=\bigoplus_{v\in E^0} U_v^\ast\Big(\lambda^1\times \lambda^0(a)\Big)U_v\\ &=\Big(\bigoplus_{v\in E^0}U_v\Big)^\ast \Big(\lambda^1\times \lambda^0(a)\Big)\Big(\bigoplus_{v\in E^0}U_v\Big)
 =U^\ast \left(\lambda^1\times \lambda^0(a)\right)U,
\end{align*}
as in equation \eqref{eq.unitaryequivalenceofproposition}. We conclude that $\Theta$ is unitarily equivalent to $\lambda^1\times\lambda^0$ and hence faithful by \cite[Lemma 4.1.7]{hawkins2015applications}.
\end{proof}

Next we describe an important set of KMS$_\infty$-states for the analysis that will follow. For each $v\in E^0$, let $\varphi_v$ be the vector state of \Toeplitz{E} given by
\begin{equation}\label{KMSinfty}
    \varphi_v(a):=\inner{\psi_v\times\pi_v(a)e_v}{e_v},\quad\text{for all } a\in \Toeplitz{E}.
\end{equation}
Denote by $S^\infty$ the set $\{\varphi_v: v\in E^0\}$.  To prove our main result, we first describe the GNS-representation of each $\varphi_v \in S^\infty$.

\begin{lemma}\label{lemma.GNSvarphiequivpsipi}
Fix $v\in E^0$. The GNS-representation of $\varphi_v\in S^\infty$ is equivalent to $\psi_v\times \pi_v$ on $\ell^2(E^\ast v)$.
\end{lemma}
\begin{proof}
 We show that $e_v$ is a cyclic vector for $\psi_v\times \pi_v$. It is clear that $\psi_v\times \pi_v(1)e_v=e_v$. Fix $f \in E^1 v$. Since $|E^1 v|<\infty$ there exists a function $x\in X(E)$ such that $x(f)=1$ and $x(g)=0$ for $g\in E^1v\setminus{\{f\}}$, hence
\[
    \psi_v\times \pi_v(\iota_{X(E)}(x))e_v = \sum_{g\in E^1 v}x(g) e_g =e_f.
\]
Thus $\{e_f : f\in E^1v\} \subset \{\psi_v\times \pi_v(a) e_v: a \in \Toeplitz{E}\}$. An induction on $n$, using a similar argument, shows that the set $\{e_\mu : \mu\in E^nv\} \subset \{\psi_v\times \pi_v(a)e_v: a \in \Toeplitz{E}\}$  for each $n \in \mathbb{N}$. Therefore $e_v$ is cyclic for $\psi_v\times \pi_v$ and this proves the lemma by the uniqueness of cyclic representations \cite[Theorem~5.1.4]{murphy1990book}. 
\end{proof}

\begin{remark}\label{rem.representationadjoint}
    If $v\in E^0$, $\xi \in X(E)$ and $\mu, \mu' \in E^\ast v$, then
    \[\inner{\psi_v(\xi)^*e_\mu}{e_{\mu'}} = \inner{e_\mu}{\sum_{f\in E^1 r(\mu')} \xi(f) e_{f\mu'}}.
    \] 
   Hence, if $\mu=f\nu$, for $f\in E^1$ and $\nu\in E^\ast v$, then $\psi_v(\xi)^* e_\mu = \overline{\xi(f)} e_\nu$. Otherwise, if $\mu=v$, then $\psi_v(\xi)^* e_\mu = 0$. In the first case, we also have the identity
    \[\psi_v(\xi)\psi_v(\xi)^* e_\mu = \overline{\xi(f)}\sum_{g\in E^1 r(\nu)} \xi(g)e_{g\nu}.\]
\end{remark}

\begin{lemma}\label{lem.sinftyelemtensor}
    Let $E$ be a topological graph and fix $v\in E^0$. Suppose that $x\in X(E)^{\otimes m}$ and $y\in X(E)^{\otimes n}$ are elementary tensors. If either $m >0$ or $n >0$, then $\varphi_v(\iota^{\otimes m}(x)\iota^{\otimes n}(y)^*) =0$.
\end{lemma}
\begin{proof}
Let $x=x_1\otimes x_2\otimes\cdots x_m\in X(E)^{\otimes m} $ and $y=y_1\otimes y_2\otimes \cdots y_n\in X(E)^{\otimes n}$ . First, suppose $n>0$. Then,
\begin{align*}
 \varphi_v(\iota^{\otimes m}(x)\iota^{\otimes n}(y)^*)&=\inner{\psi_v^{\otimes m}(x)\psi_v^{\otimes n}(y)^\ast e_v}{e_v}\\ &=\inner{\psi_v^{\otimes m}(x)\psi_v(y_n)^\ast\cdots\psi_v(y_1)^\ast e_v}{e_v} = 0,
\end{align*}

by Remark \ref{rem.representationadjoint}.
Now, suppose $n=0$ and $m>0$. Then,
\[\varphi_v(\iota^{\otimes m}(x)\iota_{C_0(E^0)}(y))=\inner{\psi_v^{\otimes m}(x)\pi_v(y)e_v}{e_v}=\inner{\psi_v(x_1)\cdots \psi_v(x_m\cdot y)e_v}{e_v}. \]
By Proposition~\ref{proposition.representationfaithful}, \[\inner{\psi_v(x_1)\cdots \psi_v(x_m\cdot y)e_v}{e_v}=\sum_{\mu = \mu_1\cdots\mu_m\in E^m v}\inner{x_1(\mu_1)\cdots x_m(\mu_m)y(v)e_\mu}{e_v}=0,\]
which concludes the proof.
\end{proof}
\begin{lemma}\label{lem.sinftyweakstarlimit}
Let $E$ be a compact topological graph. Let $(v_n)_{n\in\mathbb{N}}$ be a sequence in $E^0$. Then $\varphi_{v_n}$ is weak$^\ast$ convergent if and only if $v_n$ converges in $E^0$, in which case, writing $v:=\lim_n v_n$, we have $\varphi_{v_n}\overset{w^\ast}{\to} \varphi_v$.
\end{lemma}

\begin{proof}
    First suppose that $v_n\to v$ in $E^0$. We want to show that $\varphi_{v_n}(a) \to \varphi_v (a)$ for every $a\in \Toeplitz{E}$. By linearity and continuity it is enough to prove this for $a=\iota^{\otimes m}(x)\iota^{\otimes n}(y)^\ast$ where $x\in X(E)^{\otimes m}, y\in X(E)^{\otimes n}$ are elementary tensors. If $n>0$ or $m>0$, then $\varphi_{v_n}(a)=0 = \varphi_v(a)$ by Lemma~\ref{lem.sinftyelemtensor}, so we just need to check the case when $a=\iota_{C(E^0)}(\alpha) $ for $\alpha\in C(E^0)$. Indeed,
    \[\varphi_{v_n}(\iota_{C(E^0)}(\alpha)) = \inner{\pi_{v_n}(\alpha) e_{v_n}}{e_{v_n}} = \alpha(v_n) \to \alpha(v) = \varphi_{v}(\iota_{C(E^0)}(\alpha)).\] This proves the first direction. Now suppose that $v_n$ does not converge in $E^0$. Since $E^0$ is compact, $(v_n)$ has at least two distinct accumulation points $v,v'\in E^0$. Take $\alpha\in C(E^0)$ such that $\alpha(v)\neq \alpha(v')$. Let $(v_{n_\ell})_{\ell}$ and $(v_{n_k})_k$ be subsequences of $v_n$ that converge to $v$ and $v'$, respectively. If $\varphi_n$ weak$^\ast$ converges to some $\phi$, then by the previous paragraph,
    \[\alpha(v) = \lim_\ell \varphi_{v_{n_\ell}}(\iota_{C(E^0)}(\alpha)) = \phi(\iota_{C(E^0)}(a)) = \lim_k \varphi_{v_{n_k}}(\iota_{C(E^0)}(\alpha)) = \alpha(v'),\] a contradiction.
\end{proof}

\begin{proposition}\label{prop.sinftyweakstarlimit}
Let $E$ be a compact topological graph. The set $S^\infty$ is the set of weak$^\ast$ limit points of sequences $(\phi_n)_{n\in \mathbb{N}}$ such that there exists a sequence $\beta_n \to \infty$ of real numbers such that each $\phi_n$ is an extremal KMS$_{\beta_n}$-state.
\end{proposition}

\begin{proof}
First, assume that $(\beta_n)$ and $(\phi_n)$ are sequences as in the statement of the proposition such that $\phi_n$ weak$^\ast$ converges to some $\phi$. We need to show that there exists $v\in E^0$ such that $\phi = \varphi_v$. Since $\beta_n\to \infty$, we may assume that $\beta_n>\log(\rho(A_E))$ for all $n$. By Remark \ref{rem.extremalKMSstates}, for each $n$ there is a vertex $v_n\in E^0$ such that $\phi_n=\phi^{\beta_n}_{v_n}$, the extremal KMS$_{\beta_n}$-state described just before Theorem~\ref{theorem.topgraphkmsepsilonstate}. By equation~\eqref{eq.kmsstateinanelement}, for every $a\in \Toeplitz{E}$,
\begin{align*}
&\phi^{\beta_n}_{v_n}(a)=(N_{\phi^{\beta_n}_{v_n}})^{-1}\sum_{E^\ast v_n}e^{-\beta_n |\mu|}\inner{\psi_{v_n}\times \pi_{v_n}(a) e_\mu}{e_\mu}\\
&=(N_{\phi^{\beta_n}_{v_n}})^{-1}\Big( \inner{\psi_{v_n}\times \pi_{v_n}(a) e_{v_n}}{e_{v_n}}+ \sum_{E^\ast v_n\setminus{ \{v_n\}}}e^{-\beta_n |\mu|}\inner{\psi_{v_n}\times \pi_{v_n}(a) e_\mu}{e_\mu}          \Big).
\end{align*}

We claim that $\lim_{n \to \infty} N_{\phi^{\beta_n}_{v_n}} = 1$ and that
\[
    \lim_{n \to \infty} \sum_{E^\ast v_n\setminus{ \{v_n\}}}e^{-\beta_n |\mu|}\inner{\psi_{v_n}\times \pi_{v_n}(a) e_\mu}{e_\mu} = 0.
\]
For the first of these equalities, note that by Remark \ref{rem:maxexist}, we have for $n\in\mathbb{N}$,
\[
    1\leq N_{\phi^{\beta_n}_{v_n}}\leq 1+\sum_{i=1}^\infty e^{-\beta_n i}\max_{v\in E^0} |E^iv|.
\]
Hence $\lim_{n \to \infty} \sum_{i=1}^\infty e^{-\beta_n i}\max_{v\in E^0} |E^iv| = 0$ by the dominated convergence theorem. Thus $N_{\phi^{\beta_n}_{v_n}}\to 1$. 

The second equality follows similarly once we observe that
\[\sum_{E^\ast v_n\setminus{ \{v_n\}}}e^{-\beta_n |\mu|}\inner{\psi_{v_n}\times \pi_{v_n}(a) e_\mu}{e_\mu}\leq \|a\| \sum_{i=1}^\infty e^{-\beta_n i}\max_{v\in E^0} |E^iv|.\]

The claim implies that $\phi_{v_n}^{\beta_n} - \varphi_{v_n}$ weak$^\ast$ converges to $0$. In particular, $\varphi_{v_n}$ weak$^\ast$ converges to the same weak$^\ast$ limit as $\phi_n = \phi_{v_n}^{\beta_n}$, namely  $\phi$. So Lemma~\ref{lem.sinftyweakstarlimit} implies that $v_n$ converges to some $v\in E^0$ and $\phi = \varphi_v$.

Now, assume that $v\in E^0$. We have to show that there exist sequences $(\beta_n)$ and $(\phi_n)$ as in the statement of the proposition such that $\phi_n$ weak$^\ast$ converges to $\varphi_v$. Pick any sequence $(\beta_n)$ of real numbers such that $\beta_n > \log(\rho(A_E))$ and $\beta_n \to \infty$. Take the constant sequence $v_n = v$. By Remark \ref{rem.extremalKMSstates}, $\phi_n := \phi_{v_n}^{\beta_n} = \phi_v^{\beta_n}$ is an extremal KMS$_{\beta_n}$-state for all $n$. The same argument as before shows that $\phi_{v_n}^{\beta_n} - \varphi_{v_n} = \phi_n - \varphi_v$ weak$^\ast$ converges to $0$ and we are done.

\end{proof}
 
 \begin{lemma}\label{lemma:elementaprojection}Let $E$ be a compact topological graph. For each state $\phi$ of \Toeplitz{E}, denote its  GNS-representation by $\pi_\phi$. There exists a unique element $p_E\in \Toeplitz{E}$ such that for every $\varphi \in S^\infty$, 
 \begin{itemize}
 \item[(i)] $\pi_{\varphi}(p_E)$ is a minimal projection in $\pi_\varphi(\Toeplitz{E})$ and
 \item[(ii)] $\varphi(p_E)=1$.
 \end{itemize}
 \end{lemma}
 
 \begin{proof}
 We begin with existence. As $E^1$ is compact and the source map is a local homeomorphism, there exists a finite open cover  $(U_i)_{i=1}^n$ of $E^1$ by s-sections.  Let $\{\xi_i \}_{i=1}^n$ be a partition of unity subordinate to the $U_i$ with each $\xi_i\in C_0(U_i,[0,1])$. Writing $\sqrt{\xi_i}$ for the pointwise square-root of each $\xi_i:E^1\to [0,1]$, we define an element $p_E\in \Toeplitz{E}$ by
 \begin{equation}\label{eq:definitionofa}
    p_E:=1-\sum_{i=1}^n\iota_{X(E)}\big(\sqrt{\xi_i}\big)\iota_{X(E)}\big(\sqrt{\xi_i}\big)^\ast. 
 \end{equation}
Let $v\in E^0$  be the vertex such that $\varphi = \varphi_v$, and fix $\mu\in E^\ast v\setminus{\{v\}}$. Write $\mu = f\nu$ with $f\in E^1$. By Remark~\ref{rem.representationadjoint}, since each $\xi_i$ is supported in an s-section,
\begin{align*}\psi_v\times\pi_v(p_E)e_\mu & = e_\mu-\sum_{i=1}^n\sum_{g\in E^1 r(\nu)}\sqrt{\xi_i(f)}\sqrt{\xi_i(g)}e_{g\nu} = e_\mu - \left(\sum_{i=1}^n\xi_i(f) \right)e_\mu = 0.\end{align*}

  On the other hand, again by Remark \ref{rem.representationadjoint}, $\psi_v\times\pi_v(p_E)e_v=e_v$, so $\psi_v\times\pi_v(p_E)$ is the projection onto the basis vector $e_v$, thus by Lemma~\ref{lemma.GNSvarphiequivpsipi}, $p_E$ satisfies (i) and (ii). 
  
  Now we prove uniqueness. Fix $a\in \Toeplitz{E}$  such that $\pi_\varphi(a)$ is a minimal projection in $\pi_\varphi(\Toeplitz{E})$ and $\varphi(a)=1$ for every $\varphi\in S^\infty$. Fix $v\in E^0$. By Lemma \ref{lemma.GNSvarphiequivpsipi}, $\psi_v\times \pi_v(a)$ is a minimal projection in $\psi_v\times\pi_v(\Toeplitz{E})$. Since $\psi_v\times\pi_v(p_E) \in \psi_v\times \pi_v(\Toeplitz{E})\cap \mathcal{K}(\ell^2(E^\ast v))$, we have $\mathcal{K}(\ell^2(E^\ast v))\subset \pi_{\varphi_v}(\Toeplitz{E})$ \cite[Theorem 2.4.9]{murphy1990book}. Hence,  $\psi_v\times\pi_v(a) $ is the rank-one projection $\xi\otimes  \xi^\ast$, for a unit vector $\xi\in \ell^2(E^\ast v)$ and
\begin{align*}
   1= \varphi_v(a)=\inner{\psi_v\times \pi_v(a) e_v}{e_v}=\inner{\xi\otimes \xi^\ast (e_v)}{e_v}=\vert \inner{\xi}{e_v}\vert^2. 
\end{align*}
Since $\xi$ is a unit vector, $\xi=\lambda e_v$, for some $\lambda\in \mathbb{T}$. This implies that $\psi_v\times\pi_v(a)=\xi\otimes\xi^\ast=\lambda e_v\otimes (\lambda e_v)^\ast = e_v\otimes e_v^\ast$. Hence $\psi_v\times\pi_v(a)=\psi_v\times\pi_v(p_E)$ for every $v\in E^0$. Proposition \ref{proposition.representationfaithful} implies that $\bigoplus_{v\in E^0}\psi_v\times \pi_v$ is faithful, so  $a=p_E$.
 \end{proof}

 \begin{remark}\label{remark.M=Ma}
Direct calculations with basis vectors, similar to those in the proof of Lemma \ref{lemma:elementaprojection}, show that the faithful representation $\lambda^0\times \lambda^1$ of \cite[Lemma 4.1.7]{hawkins2015applications} defined in the proof of Proposition \ref{proposition.representationfaithful} carries $p_E$ to the projection onto $\clsp\{e_v:v\in E^0\}\subset \ell^2(E^\ast)$. Define $M\coloneqq \iota_{C(E^0)}(C(E^0))$ and observe that $p_E\in M'$, the commutant of $M$. Indeed, for $\alpha\in C(E^0)$ and $\mu\in E^\ast$,
\begin{align*}
\lambda^0\times \lambda^1(\iota_{C(E^0)}(\alpha)p_E)e_\mu
    &= \begin{cases}
        \lambda^0(\alpha)e_\mu & \text{if $|\mu|=0$}\\
        0& \text{if $|\mu|> 0$}
        \end{cases}\\
    &= \lambda^0\times \lambda^1(p_E\iota_{C(E^0)}(\alpha))e_\mu.
\end{align*}
 In particular, if $\iota_{C(E^0)}(\alpha)p_E = 0$, then $\lambda^0(\alpha)e_v = 0$ for all $v\in E^0$ so that $\alpha = 0$. Thus the map $\rho_E: M\to Mp_E$ given by $\rho(m)\coloneqq mp_E=p_E m$ is an isomorphism of $C^\ast$-algebras.
 \end{remark}
In what follows, we write $\iota_{X(E)}(\xi)$ as $\iota_X(\xi)$ to lighten notation.

For the following lemma, recall that $\mathcal{T}C^*(E)_1$ denotes the first spectral subspace $\{a \in \mathcal{T}C^*(E) : \gamma_z(a) = za\text{ for all } z \in \mathbb{T}\}$.

 \begin{lemma}\label{lemma.ismorphismofCstarcorrespondences}
 Let $E$ be a compact topological graph. Let $M:=\iota_{C(E^0)}(C(E^0))\subset \Toeplitz{E}$, let $p_E\in \Toeplitz{E}$ be as in Lemma \ref{lemma:elementaprojection} and let $\rho_E:M\to Mp_E$ be the isomorphism $\rho_E(m)=mp_E$ of Remark~\ref{remark.M=Ma}. Define 
 \[
 \psi_E: X(E)\to \leftidx{_\Toeplitz{E}}{\Toeplitz{E}}{_\Toeplitz{E}}
 \]
 by $\psi_E(\xi)=\iota_{X}(\xi)p_E$. Then $(\psi_E,\iota_{C(E^0)})$ is a bimodule morphism, and $\psi_E(X(E))=\Toeplitz{E}_1p_E$. Furthermore, there is an $M$-valued inner product on $\Toeplitz{E}_1p_E$ such that
 \begin{equation}\label{eq:M-valued ip}
     \inner{\psi_E(\xi)}{\psi_E(\eta)}_M=\rho_E^{-1}(\psi_E(\xi)^\ast\psi_E(\eta))\text{ for all } \xi,\eta\in X(E),
 \end{equation}
 and with respect to this inner product, $(\psi_E,\iota_{C(E^0)}):X(E)\to \Toeplitz{E}_1p_E$ is an isomorphism of $C^\ast$-correspondences.
 \end{lemma}
 \begin{proof}
 We start by showing that $\psi_E(X(E))=\Toeplitz{E}_1p_E$. For this, recall that $\Toeplitz{E}=\clsp\{\iota_X^{\otimes m}(x)\iota_X^{\otimes n}(y)^\ast: n,m\in \mathbb{N},x\in X(E)^{\otimes m},y\in X(E)^{\otimes n}\}$. Note that $\gamma_z(\iota_X^{\otimes m}(x)\iota_X^{\otimes n}(y)^\ast)=z^{m-n}\iota_X^{\otimes m}(x)\iota_X^{\otimes n}(y)^\ast$, so
 \[\Toeplitz{E}_1=\clsp\{\iota_X^{\otimes n+1}(x)\iota_X^{\otimes n}(y)^\ast: n\geq 0, x\in X(E)^{\otimes n+1},y\in X(E)^{\otimes n}\}.\]
 Hence,
 \[\Toeplitz{E}_1p_E=\clsp\{\iota_X^{\otimes n+1}(x)\iota_X^{\otimes n}(y)^\ast p_E: n\geq 0, x\in X(E)^{\otimes n+1},y\in X(E)^{\otimes n}\}.\]
 Suppose that $x=x_1\otimes x_2\otimes...\otimes x_{n+1}$ and that $y=y_1\otimes y_2\otimes...\otimes y_n$ if $n>0$ or $y=a\in\leftidx{_A}{A}{_A}$ if $n=0$ . For $\mu\in E^\ast$, the representation $(\lambda^0,\lambda^1)$ as in Remark \ref{remark.M=Ma} satisfies
 \[\lambda^0\times\lambda^1(\iota_X^{\otimes n+1}(x)\iota_X^{\otimes n}(y)^\ast p_E)e_\mu=\begin{cases}\lambda^0\times\lambda^1(\iota_X^{\otimes n+1}(x)\iota_X^{\otimes n}(y)^\ast)e_\mu,\text{ if $|\mu|=0$},\\0, \text{ if $|\mu|>0$.}
 \end{cases}
\]
If $|\mu|=0$, then by the proof of \cite[Lemma 4.1.7]{hawkins2015applications}, if $n>0$, we have $\lambda^1(y_1)^\ast e_\mu=0$. Hence, for $|\mu|=0$,
\[\lambda^0\times\lambda^1(\iota_X^{\otimes n+1}(x)\iota_X^{\otimes n}(y)^\ast)e_\mu=\begin{cases}\lambda^1(x\cdot a^\ast)e_\mu, \text{ if $n=0$},
\\
0, \text{ if $n>0$}.
\end{cases}\]
Since $\lambda^1\times\lambda^0$ is faithful representation, we conclude that $\iota_X^{\otimes n+1}(x)\iota_X^{\otimes n}(y) p_E=0$ if $n>0$ and
\begin{align*}
   \Toeplitz{E}_1p_E&=\clsp\{\iota_X(x)\iota_{C(E^0)}(a)^\ast p_E: x\in X(E),a\in C(E^0)\}\\
   &=\{\psi_E(\xi): \xi\in X(E)\}=\psi_E(X(E)). 
\end{align*}
 
  Now we show that $(\psi,\iota_{C(E^0)})$ is a bimodule morphism. Indeed, for $a\in C(E^0)$ and $\xi\in X(E)$,
\[\psi_E(a\cdot\xi)=\iota_X(a\cdot \xi)p_E=\iota_{C(E^0)}(a)\iota_X(\xi)p_E=\iota_{C(E^0)}(a)\psi_E(\xi).\]
Since $p_E$ commutes with $M$ as in Remark \ref{remark.M=Ma}, it follows from a similar computation that $\psi$ preserves the right actions of $C(E^0)$. So $(\psi_E,\iota_{C(E^0)})$ preserves the bimodule structure. It remains to show that $(\psi_E,\iota_{C(E^0)})$ is an isomorphism of Hilbert modules. A simple calculation shows that it respects the $M$-valued inner-product of~\eqref{eq:M-valued ip}:
 \begin{align*}
   \inner{\psi_E(\xi)}{\psi_E(\eta)}_M
    &= \rho^{-1}_E(\psi_E(\xi)^\ast\psi_E(\eta))\\
    &=\rho^{-1}_E((\iota_X(\xi)a)^\ast\iota_X(\eta)p_E)
     =\iota_X(\xi)^\ast\iota_X(\eta)
     =\iota_{C(E^0)}(\inner{\xi}{\eta}_{C(E^0)}).  
 \end{align*}
 Hence $X(E)\cong \Toeplitz{E}_1p_E$ as $C^\ast$-correspondences.
 \end{proof}

We can now prove our first main theorem. Recall the notion of an isomorphism of triples from the introduction.
 
 \begin{theorem} Let $E$ and $F$ be compact topological graphs. Let $\gamma^E$ and $\gamma^F$ be the gauge actions on \Toeplitz{E} and \Toeplitz{F} and let $M_E=\iota_{C(E^0)}(C(E^0))$ and $M_F=\iota_{C(F^0)}(C(F^0))$. Suppose that $\theta:(\Toeplitz{E},\gamma^E,M_E)\to(\Toeplitz{F},\gamma^F,M_F)$ is an isomorphism of triples. Let $\theta_M:=\theta|_{M_E}$. Then there exists a unique linear map $\theta_X:X(E)\to X(F)$ such that 
 \begin{equation}\label{eq.theta_X}
\theta(\iota_X(\xi)p_E)=\iota_X(\theta_X(\xi))p_F\quad \text{for all }\xi\in X(E),     
 \end{equation}
 and $(\theta_X,\theta_M)$ is an isomorphism of $C^*$-correspondences. 
 \end{theorem}
 \begin{proof}
 We first claim that $\theta(\Toeplitz{E}_1)=\Toeplitz{F}_1$ and $\theta(p_E)=p_F$. The first is because $\theta$ interwines the gauge actions $\gamma_E$, $\gamma_F$. For the second we show that $\theta(p_E)\in \Toeplitz{F}$ satisfies properties (i)~and~(ii) of Lemma~\ref{lemma:elementaprojection}. Let $\varphi\in S^\infty_F$ (the subscript $F$ is to make it explicit that we are referring to the states of \Toeplitz{F}). Then writing $\xi$ for the corresponding cyclic vector in the GNS-space $\mathcal{H}_\varphi$, for any $b\in \Toeplitz{E}$,
 \[\varphi(\theta(b))=\inner{\pi_\varphi(\theta(b))\xi}{\xi}.\]
Since $\theta$ is a gauge invariant isomorphism, $\phi\to \phi\circ \theta$ is an isomorphism of the KMS$_\beta$-simplex of $(\Toeplitz{F},\gamma_F)$ to that of $(\Toeplitz{E},\gamma_E)$ for each $\beta$. Thus, given an extremal KMS$_\beta$-state $\phi$ of \Toeplitz{F}, the composition $\phi\circ \theta$ is an extremal KMS$_\beta$-state of $\Toeplitz{E}$. Since $p_E$ satisfies condition~(ii) of Lemma~\ref{lemma:elementaprojection} in $\mathcal{T}C^*(E)$, we have $\phi(\theta(p_E)) = \phi \circ \theta(p_E) = 1$ for each extremal $\phi$, and so $\theta(p_E)$ satisfies condition~(ii) of Lemma~\ref{lemma:elementaprojection} in $\mathcal{T}C^*(F)$. So, writing $\eta$ for the cyclic vector in the GNS-space $\mathcal{H}_{\varphi\circ \theta}$, for $b\in \Toeplitz{E}$ we have
 \[\varphi\circ\theta(b)=\inner{\pi_{\varphi\circ \theta}(b)\eta}{\eta}.\]
 Hence, by uniqueness of the GNS construction, $\pi_{\varphi}\circ \theta$ is unitarily equivalent to $\pi_{\varphi\circ \theta}$. Since $\pi_\varphi(\theta(b))$ is a minimal projection if and only if $\pi_{\varphi\circ\theta}(b)$ is a minimal projection, $\theta(p_E)$ satisfies condition~(i) of Lemma~\ref{lemma:elementaprojection}. Thus by uniqueness, $\theta(p_E)=p_F$. 
 
 Considering the diagram
 
 \[\tag{$\star$}
\begin{tikzcd}
X(E) \arrow[d, "\theta_X"', dashed] \arrow[r, "\psi_E"] & \Toeplitz{E}_1p_E \arrow[r, hook] \arrow[d, "\theta"] & \Toeplitz{E} \arrow[d, "\theta"] \\
X(F) \arrow[r, "\psi_F"']                               & \Toeplitz{F}_1p_F \arrow[r, hook]                     & \Toeplitz{F},                    
\end{tikzcd}
\]
 we show that the map $\theta_X$ that makes the diagram $\star$ commute satisfies \eqref{eq.theta_X}. For $\xi\in X(E)$,
\begin{align}
    \iota_{X(F)}(\theta_X(\xi))p_F&=\iota_{X(F)}(\psi_F^{-1}\circ\theta\circ\psi_E(\xi))p_F \label{eq.prooftheorem}\\
    &=\iota_{X(F)}(\psi_F^{-1}\circ\theta(\iota_{X(E)}(\xi)p_E))p_F.\nonumber
\end{align}
We already saw that $\theta(\Toeplitz{E}_1)=\Toeplitz{F}_1$ and $\theta(p_E)=\theta(p_F)$. Consequently, $\theta(\iota_{X(E)}(\xi)p_E)=\theta(\iota_{X(E)}(\xi))p_F\in \Toeplitz{F}_1p_F$, so there exists $\xi'\in X(F)$ such that
\[\theta(\iota_{X(E)}(\xi)p_E)=\iota_{X(F)}(\xi')p_F=\psi_F(\xi').\]
Hence, we obtain
\[\iota_{X(F)}(\theta_X(\xi)p_F)\overset{\eqref{eq.prooftheorem}}{=}\iota_{X(F)}(\xi')p_F=\theta(\iota_{X(E)}(\xi)p_E),\] that is \eqref{eq.theta_X}.
Since the maps $\psi_E$ and $\psi_F$ in the diagram are isomorphisms, we deduce that $(\theta_X,\theta_M)$ is an isomorphism of $C^\ast$-correspondences.
 \end{proof}

\section{Local reconstruction of topological graphs}\label{section.local}

In this section we investigate how to recover a compact topological graph up to local conjugacy in the sense of Davidson--Roydor \cite{DavRoy} from its Hilbert bimodule. We arrived at this result independently, but subsequently discovered that it can be recovered from Davidson and Roydor's results about tensor algebras. We thank both Adam Dor-On and the anonymous referee for bringing these results to our attention. We include a proof here because we feel that the direct passage from the bimodule to the local conjugacy class of the graph is more elementary than the approach that passes through the tensor algebra. 

We first recall Davidson and Roydor's notion of local conjugacy, which in turn is based on Davidson and Katsoulis' notion of local conjugacy for local homeomorphisms \cite{DavKat}.

\begin{definition}[{\cite[Definition~4.3]{DavRoy}}]\label{definition.localisomorphism}
Let $E$ and $F$ be topological graphs. We say that $E$ and $F$ are \textit{locally conjugate}, and write $E \cong_{\operatorname{loc}} F$ if there exists a homeomorphism $\phi^0:E^0\to F^0$ such that for each  $v\in E^0$ there is a neighborhood $U$ of $v$ and a homeomorphism $\phi^1_U:E^1U\to F^1\phi^0(U)$ such that 
\begin{equation}\label{eq.localisomorphismdefinition}
r_F\circ \phi_U^1=\phi^0\circ r_E|_{E^1 U}\qquad\text{and}\qquad 
s_F\circ \phi_U^1=\phi^0\circ s_E|_{E^1 U}.    
\end{equation}
\end{definition}

Our main result in this section is that if compact topological graphs have isomorphic graph modules, then they are locally conjugate. For that, we need to collect some technical lemmas first.
\begin{lemma}\label{lemma:permutationofbasisvectors}
Let $(e_1, \dots, e_k)$ be the standard basis for $\mathbb{C}^k$. If $\{x_1,...,x_k\}\subset \mathbb{C}^k$ is a basis, there exists a permutation $\sigma$ of $\{1, \dots, k\}$ such that $\inner{x_i}{e_{\sigma(i)}}\neq 0$ for all $i=1,...,k$.
\end{lemma}
\begin{proof}
Let $A$ be the $k\times k$ matrix with $i$th column $x_i$, and let $S_k$ be the symmetric group. Since $\{x_1,...,x_k\}$ are linearly independent, $A$ is invertible. Therefore,
\[\det(A)=\sum_{\sigma\in S_k} \big(\operatorname{sgn}(\sigma)\prod_{i=1}^k A_{i,\sigma(j)}\big)\neq 0.\]
Hence there exists $\sigma\in S_k$ such that $\prod_{i=1}^k A_{i,\sigma(i)}\neq 0$, and hence \[\inner{x_i}{e_{\sigma(i)}} = A_{i,\sigma(i)}\neq 0\] for all $i=1,...,k$.
\end{proof}

\begin{lemma}\label{lemma.principallemmatoprovetheorem}
Let $E$ be a compact topological graph. Suppose that $h\in C(E^0,[0,1])$ and $g_1,\ldots,g_k\in X(E)$ satisfy
\begin{itemize}
    \item[(1)]$\inner{g_i}{g_j}_{C(E^0)}=\delta_{i,j} h$; 
    \item[(2)]$\overline{X(E)\cdot h}\subseteq\overline{\operatorname{span}}\{g_i\cdot a: i\leq k, a\in C(E^0)\}$; and
    \item[(3)] For each $i\leq k$, there exists a continuous function $\alpha_i: \supp(h) \to E^0$ such that for every $a \in C(E^0)$ and for any function $\tilde{a} \in C(E^0)$ such that $\tilde{a}|_{\supp(h)} = a \circ \alpha_i$, we have $a \cdot g_i = g_i \cdot \tilde{a}$.
\end{itemize}
Then, for each $v\in \operatorname{supp}^\circ(h)$, there exist a neighborhood $W \subseteq \supp(h)$ of $v$, s-sections $Z_e$, $e\in E^1 v$ as in Lemma~\ref{lemma:s-sectionslocalconstant} and a bijection $\sigma:\{1,\ldots,k\}\to E^1 v$ such that each $\alpha_i=r\circ (s|_{Z_{\sigma(i)}})^{-1}$ on $W$. 
\end{lemma}
 \begin{proof}
 Fix $v\in \operatorname{supp}^\circ(h)$. For $i\le k$, regard $g_i|_{E^1v}$ as an element of $\ell^2(E^1v)$. Then for $i,j \le k$,
\begin{align*}
\inner{g_i|_{E^1 v}}{g_j|_{E^1v}}
    &= \sum_{e\in E^1v}\overline{g_i|_{E^1v}(e)}g_j|_{E^1v}(e)= \inner{g_i}{g_j}(v)
    = \delta_{i,j} h(v).
\end{align*}
Since $h(v)\not= 0$ it follows that $\{g_i|_{E^1v} : i \le k\}$ is linearly independent. We claim that $\lsp\{g_i|_{E^1v}: i \le k\} = \ell^2(E^1v)$. For this, note that $\operatorname{res} : \xi \mapsto \xi|_{E^1v}$ is a norm-decreasing linear map from $X(E) \cdot h$ to $\ell^2(E^1v)$, and is surjective since $h(v)\not= 0$. So condition~(2) gives 
\begin{align*}
\ell^2(E^1v) 
    &\subseteq \operatorname{res}\big(\overline{\operatorname{span}}\{g_j\cdot a: j\leq k, a\in C(E^0)\}\big) \\
    &\subseteq \clsp\{a(v)g_i|_{E^1 v} : a \in C_0(E^0), i \le k\} 
    = \lsp\{g_i|_{E^1 v} : i \le k\}.
\end{align*}
Hence $|E^1 v|=k$, and so Lemma~\ref{lemma:permutationofbasisvectors} yields a bijection $\sigma:\{1,\ldots,k\}\to E^1 v$ such that $g_i(\sigma(i))\neq 0$ for all $i=1,\ldots,k$.

We claim that whenever $g_i(e) \not= 0$, the function $\alpha_i$ of~(3) satisfies $\alpha_i(s(e))=r(e)$. To see this, we prove the contrapositive. So suppose that $\alpha_i(s(e)) \not= r(e)$. By Tietze's theorem there exists $a \in C(E^0)$ such that $a(r(e))=1$ and $a(\alpha_i(s(e)))=0$. The function $a \circ \alpha_i|_{\supp{h}}$ is a continuous function on the compact set $\supp(h)$, so by Tietze's theorem there is a function $\tilde{a} \in C(E^0)$ that extends $a \circ \alpha_i|_{\supp{h}}$. Since $s(e) \in \supp(h)$ we have $\tilde{a}(s(e)) = a(\alpha_i(s(e)))$, and so condition~(3) gives
\[
    0 = g_i(e) a(\alpha_i(s(e))) 
      = (g_i \cdot \tilde{a})(e)
      = (a\cdot g_i)(e) = a(r(e))g_i(e) 
      = g_i(e),
\]
proving the claim.

Fix a neighborhood $W$ of $v$ and s-sections $Z_e$, $e\in E^1 v$ as in Lemma~\ref{lemma:s-sectionslocalconstant}. Since the $g_i$ are continuous and each $g_i(\sigma(i))\neq 0$, for each $i$ there is a neighbourhood $\sigma(i) \in E^1$ on which $g_i$ is everywhere nonzero. Shrinking the $Z_{\sigma(i)}$ and $W$ appropriately, we may therefore assume that $Z_{\sigma(i)}\subseteq \operatorname{supp}g_i$ for each $i$ and $W \subseteq \supp(h)$. Fix $w\in W$. Then $(s|_{Z_{\sigma(i)}})^{-1}(w) \in \operatorname{supp}g_i$ and by the claim above
$\alpha_i(w) = \alpha_i(s\circ (s|_{Z_{\sigma(i)}})^{-1}(w)) = r((s|_{Z_{\sigma(i)}})^{-1}(w))$ as needed.
\end{proof}

\begin{remark}
The combination of (1)~and~(2) of Lemma~\ref{lemma.principallemmatoprovetheorem} actually implies the stronger condition~($2'$) that $\overline{X(E)\cdot h} = \overline{\operatorname{span}}\{g_j\cdot a: j\leq k, a\in C(E^0)\}$. To see this, recall from the proof of \cite[Proposition~2.31]{raeburn1998morita} that each $x \in X(E)$ can be written as $x = y \cdot \langle y, y\rangle_{C_0(E^0)}$ where $y$ is the upper right-hand entry of the matrix $\big(\begin{smallmatrix} 0\; & x \\ x^* & 0\end{smallmatrix}\big)^{1/3}$. Applying this with $x = g_j$ we obtain 
\[
\langle y, y\rangle_{C_0(E^0)} = \langle g_j,  g_j\rangle_{C_0(E^0)}^{1/3} = h^{1/3} \in \overline{C_0(E^0)h},
\]
and so $g_j = y \cdot \langle y, y\rangle_{C_0(E^0)} \in X(E) \cdot \overline{C_0(E^0)h} \subseteq \overline{X(E) \cdot h}$. But since~(2) is easier to check than~($2'$) we have stated the lemma with the formally weaker hypothesis.
\end{remark}

 \begin{theorem}\label{theorem:localisomorphismoftopologicalgraphs}
 Let $E$ and $F$ be compact topological graphs. Suppose that $X(E)\cong X(F)$ as Hilbert bimodules. Then $E \cong_{\operatorname{loc}} F$.
 \end{theorem}
 \begin{proof}
Let $\theta=(\theta^0,\theta^1)$ be a bimodule isomorphism from $X(F)$ to $X(E)$, so  $\theta^1:X(F)\to X(E)$ and $\theta^0:C(F^0)\to C(E^0)$ preserve the bimodule structure. Let $\phi^0=\widehat{\theta^0}:E^0\to F^0$ be the Gelfand transform of $\theta^0$ which is a homeomorphism between $E^0$ and $F^0$. Fix $v\in E^0$. Applying Lemma~\ref{lemma:s-sectionslocalconstant} to $\phi^0(v)$ we obtain an open neighborhood $W'\subseteq F^0$ of $\phi^0(v)$ and s-sections $Z_e'$ for $e \in F^1\phi^0(v)$ such that $F^1W' = \bigsqcup_{e\in F^1\phi^0(v)}Z_e'$. Since $F^0$ is normal, Urysohn's lemma gives a function $h'\in C_c(W',[0,1])$ (in particular, the compact set $\supp(h)$ is contained in $W'$) such that $h'(\phi^0(v))=1$. For each $e\in F^1\phi^0(v)$, define $g_e'\in C_0(Z_e')\subseteq X(F)$ by $g_e'=\sqrt{h'\circ s|_{Z'_e}}$. Define $\alpha_e':\supp(h') \to F^0$ by $\alpha_e'=r_F\circ (s_F|_{Z_e'})^{-1}$. Direct computation shows that $h'$, the $g_e'$ and the $\alpha_e'$ satisfy condition~(1) of Lemma~\ref{lemma.principallemmatoprovetheorem}. To see that they also satisfy condition~(2), fix $\xi \in X(F) = C(F^1)$ and $e \in F^1\phi^0(v)$. If $f\in Z_e'$. Then
\begin{align*}(\xi\cdot h')(f) & = \xi((s_F|_{Z_e'})^{-1}(s_F(f)))\sqrt{h'(s_F(f))}\sqrt{h'(s_F(f))} \\ 
& = \Big(g'_e(f)\cdot\big(\xi\circ(s_F|_{Z'_e})^{-1} \sqrt{h'}\big)\Big)(f).
\end{align*} 

For each $e \in F^1\phi^0(v)$, let $a_e := \big(\xi\circ(s_F|_{Z_e'})^{-1} \sqrt{h'}\big) \in C_0(W') \subseteq C(F^0)$. Since the $Z_e'$ are disjoint and $\operatorname{supp}(\xi\cdot h') \subseteq F^1W' = \bigsqcup_{e\in F^1\phi^0(v)}Z_e'$, we can write $\xi\cdot h' = \sum_{e\in F^1\phi^0(v)} g_e' \cdot a_e$. This gives condition~(2). For condition~(3), fix $a' \in C(F^0)$ and suppose that $\tilde{a}' \in C(F^0)$ extends $a' \circ \alpha'_e$. We must show that $(a' \cdot g'_e)(f) = (g'_e \cdot \tilde{a}')(f)$ for all $f \in F^1$. We consider two cases. First suppose that $f \not\in \supp(g'_e)$, then $(a' \cdot g'_e)(f) = a'(r(f))g'_e(f) = 0 = g'_e(f) \tilde{a}'(s(f)) = (g'_e \cdot \tilde{a}')(f)$. Now suppose that $f \in \supp(g'_e)$. Then $f \in Z'_e$, which implies that $\alpha'_e(s(f)) = r(f)$; and $s(f) \in \supp(h')$, giving $\tilde{a}'(s(f)) = a'(\alpha'_e(s(f))) = a'(r(f))$. Hence $(a' \cdot g'_e)(f) = a'(r(f))g'_e(f) = \tilde{a}'(s(f))g'_e(f) = (g'_e \cdot \tilde{a}')(f)$.

Let $h:=\theta^0(h')=h'\circ \phi^0$ and for each $e\in F^1 \phi^0(v)$, let $g_e:=\theta^1(g_e')$ and $\alpha_e=(\phi^0)^{-1}\circ \alpha'_e\circ \phi^0$, defined on $(\phi^0)^{-1}(\supp(h')) = \supp(h) \subseteq E^0$. Now we claim that $h$, the $g_e$ and the $\alpha_e$ also satisfy (1)--(3) of Lemma~\ref{lemma.principallemmatoprovetheorem}. Since $(\theta^0,\theta^1)$ is a bimodule isomorphism, conditions (1)~and~(2) follow from straightforward calculations using that $h'$, the $g'_e$ and the $\alpha'_e$ satisfy (1)~and~(2). We show that condition~(3) holds. For this, fix $a \in C(E^0)$ and let $a' := (\theta^0)^{-1}(a) \in C(F^0)$. Suppose that $\tilde{a} \in C(E^0)$ extends $a \circ \alpha_e$. We claim that $\tilde{a}' := (\theta^0)^{-1}(\tilde{a})$ extends $a' \circ \alpha'_e|_{\supp(h')}$. To see this, note first that, by definition of $h$, we have $\supp(h) = (\phi^0)^{-1}(\supp(h'))$. Also, by definition of $\alpha_e$, we have $\alpha_e \circ (\phi^0)^{-1} = (\phi^0)^{-1} \circ \alpha'_e$. So for $w \in \supp(h')$,
\begin{align*}
(\theta^0)^{-1}(\tilde{a})(w) 
    &= \tilde{a}((\phi^0)^{-1}(w))
    = (a \circ \alpha_e)((\phi^0)^{-1}(w))\\
    &= (a \circ (\phi^0)^{-1}) \circ \alpha'_e(w)
    = a' \circ \alpha'_e(w).
\end{align*}
Consequently, condition~(3) for $h'$, $g'_e$ and $\alpha'_e$ gives $a' \cdot g'_e = g'_e \cdot \tilde{a}'$. Hence 
\begin{align*}
a\cdot g_e 
    &= \theta^1(a'\cdot g_e') 
     = \theta^1(g_e' \cdot \tilde{a}')
     = \theta^1(g'_e)\cdot \theta^0(\tilde{a}')
     = g_e \cdot \tilde{a}.
\end{align*} 
Thus condition~(3) holds for $h$, the $g_e$ and the $\alpha_e$.

As $h(v) = h'(\phi^0(v)) = 1$, we have $v\in \supp^\circ(h)$, so by Lemma \ref{lemma.principallemmatoprovetheorem}, there exists an open neighborhood $W \subseteq (\phi^0)^{-1}(W')$ of $v$ and s-sections $Z_f$, $f\in E^1v$ as in Lemma \ref{lemma:s-sectionslocalconstant}, and a bijection $\sigma : F^1\phi^0(v) \to E^1v$ such that $\alpha_e = r_E\circ (s_E|_{Z_{\sigma(e)}})^{-1}$ on $W$ for each $e\in F^1\phi^0(v)$.

Define $\phi_W^1: E^1 W\to F^1 \phi^0(W)$ by $\phi^1_W|_{Z_{\sigma(e)}}=(s_F|_{Z'_e})^{-1}\circ \phi^0\circ s_E$ for each $e\in F^1\phi^0(v)$. Since each $\phi^1_W|_{Z_{\sigma(e)}}$ is a homeomorphism and their images are disjoint, $\phi^1_W$ is a homeomorphism by the pasting lemma. We claim that
 $\phi^1_W$ satisfies the conditions necessary for local conjugacy, that is, $r_F\circ \phi_W^1=\phi^0\circ r_E|_{E^1 W}$  and $s_F\circ \phi_W^1=\phi^0\circ s_E|_{E^1 W}$. That $s_F\circ \phi_W^1=\phi^0\circ s_E|_{E^1 W}$ is by definition of $\phi^1_W$. Now, fix $\bar{f}\in E^1 W$. There exists a unique $f\in E^1 v$ such that $\bar{f}\in Z_{f}$. Since $\sigma$ is a bijection there is a unique $e\in F^1\phi^0(v)$ such that $\sigma(e)=f$, hence $\bar{f}\in Z_{\sigma(e)}$. Thus we have
 \begin{align*}
     \phi^0\circ r_E(\bar{f})&=\phi^0(r_E((s|_{Z_{\sigma(e)}})^{-1}(s_E(\bar{f}))))=\phi^0(\alpha_e(s_E(\bar{f})))\\
     &= \phi^0((\phi^0)^{-1}\circ \alpha'_e\circ \phi^0(s_E(\bar{f})))=\alpha_e'(\phi^0(s_E(\bar{f})))\\
     &=r_F((s_F|_{Z_e'})^{-1}(\phi^0(s_E(\bar{f}))))=r_F((s_F|_{Z_e'})^{-1}(s_F(\phi^1_W(\bar{f})))).
 \end{align*}
Since $\bar{f}\in Z_{\sigma(e)}$, we have $\phi^1_W(\bar{f})=(s_F|_{Z_e'})^{-1}(\phi^0\circ s_E(\bar{f}))\in Z_e'$, therefore $\phi^0\circ r_E(\bar{f})=r_F(\phi^1_W(\bar{f}))$. Hence $E$ and $F$ are locally conjugate.
 \end{proof}

 \begin{remark}
     Theorem~\ref{theorem:localisomorphismoftopologicalgraphs} is a corollary of \cite[Theorem~4.5]{DavRoy}: if $E$ and $F$ have isomorphic bimodules, then the tensor algebras of these bimodules are also isomorphic, and so Davidson and Roydor's result implies that $E$ and $F$ are locally conjugate.
 \end{remark}

Our next corollary says that for topological graphs with totally disconnected vertex spaces, local conjugacy coincides with isomorphism. This is a consequence of the general theorem \cite[Theorem~5.5]{DavRoy}, which says that local conjugacy coincides with isomorphism whenever the covering dimension of the vertex-spaces of the graphs involved is at most~1. We have included a proof since the argument is simpler for zero-dimensional spaces.
 
 \begin{corollary}\label{corollary.isomorphismoftotallydisconnectedtopgraph}
 Let $E$ and $F$ be compact topological graphs and suppose that $E^0$ is totally disconnected. Let $\gamma^E$ and $\gamma^F$ be the gauge actions on \Toeplitz{E} and \Toeplitz{F} and let $M_E=\iota_{C(E^0)}(C(E^0))$ and $M_F=\iota_{C(F^0)}(C(F^0))$. Suppose that $(\Toeplitz{E},\gamma^E,M_E)\cong (\Toeplitz{F},\gamma^F,M_F)$. Then $E$ and $F$ are isomorphic as topological graphs.
 \end{corollary}
 \begin{proof}
 Theorems~\ref{theorem.isoKMScompacttopgraph} and \ref{theorem:localisomorphismoftopologicalgraphs} imply that $E$ and $F$ are locally conjugate.  Let $\phi^0 : E^0\to F^0$ be the map on vertices implementing the local conjugacy. Then for each $v\in E^0$ there exists a neighborhood $U_v$ of $v$ and a homeomorphism
    $\phi^1_{U_v}:E^1U_v\to F^1\phi^0(U_v)$
    such that $r_F\circ \phi_{U_v}^1=\phi^0\circ r_E|_{E^1U_v}$ 
    and $s_F\circ \phi_{U_v}^1=\phi^0\circ s_E|_{E^1U_v}$. Since $E^0$ is totally disconnected, we may suppose the $(U_v)$ are compact open. The $(U_v)_v$ cover $E^0$, so admit a finite subcover $(U_i)_{i=1}^n$. Put $V_1=U_1$ and let $V_i=U_i\setminus{\bigcup_{j=1}^{i-1} U_j}$ for $i\geq 2$. Then the $V_i$ are mutually disjoint compact open sets that cover $E^0$ and we can define the homeomorphisms $\phi^1_{V_i}:E^1 V_i\to F^1 \phi^0(V_i)$ by the restriction of $\phi^1_{U_i}$ on $E^1V_i$. Note that $(E^1V_i)_i$ are mutually disjoint open sets that cover $E^1$ and, similarly, $(F^1\phi^0(V_i))_i$ are mutually disjoint open sets that cover $F^1$ . We define $\phi^1:E^1\to F^1$ by $\phi^1|_{E^1 V_i}=\phi^1_{V_i}$. By the pasting lemma $\phi^1$ is continuous and bijective, thus a homeomorphism. It only remains to verify that $r_F\circ \phi^1=\phi^0\circ r_E$ and $s_F\circ \phi^1=\phi^0\circ s_E$. We show this only for the range map, since it is analogous for the source map. If $f\in E^1$, then $f\in E^1V_i$ for some $i=1,\ldots,n$ and hence $(r_F\circ \phi^1) (f)=r_F(\phi^1_{V_i}(f))=(\phi^0\circ r_E)(f)$.
 \end{proof}

\begin{corollary}
Let $E$ and $F$ be compact topological graphs. If $X(E)\cong X(F)$ as Hilbert bimodules, then $E \cong F$ as discrete directed graphs, via an isomorphism that implements a homeomorphism $E^0 \cong F^0$.
\end{corollary}
\begin{proof}
By Theorem~\ref{theorem:localisomorphismoftopologicalgraphs}, $E\cong_{\operatorname{loc}} F$. Hence there is a homeomorphism $\phi^0:E^0\to F^0$ as in Definition~\ref{definition.localisomorphism}. Fix $v\in E^0$. There is a neighborhood $U$ of $v$ and a homeomorphism $\phi^1_U : E^1 U \to F^1 \phi^0(U)$ satisfying equations \eqref{eq.localisomorphismdefinition}. Hence, for $w\in E^0$ we have $\phi^1_U(wE^1v)=\phi^0(w)F^1\phi^0(v)$. Since $\phi_U^1$ is a bijection, $|wE^1v|=|\phi^0(w)F^1\phi^0(v)|$. Since $v,w$ were arbitrary, it follows that there is a range-and-source-preserving bijection $\phi^1 : E^1 \to F^1$ such that $(\phi^0, \phi^1)$ is an isomorphism of discrete graphs.
\end{proof}

\section{Example}\label{section.example}
In this section we describe two nonisomorphic topological graphs whose graph correspondences are isomorphic (so that, in particuar, the graphs are locally conjugate). This proves that a generalization of Theorem~\ref{theorem:localisomorphismoftopologicalgraphs} to arbitrary compact topological graphs is not possible. 

Examples of locally conjugate local homeomorphisms that are not conjugate (so that the associated topological graphs are not isomorphic) appear in \cite[Example~3.18]{DavKat}. Furthermore, since the graphs described below have covering dimension equal to~1, once we have established that they are locally conjugate, we could deduce that their Hilbert modules are isomorphic from \cite[Theorem~5.5]{DavRoy}. However, as our example is explicit and we are able to describe an explicit isomorphism of the Hilbert modules of the topological graphs involved, we present the details.

Let $E^0=F^0=F^1=\mathbb{T}$, and let $E^1=\mathbb{T}\times\{0,1\}$. Define range and source maps in $E$,$F$ by $r_E(z,j)=s_E(z,j)=z$ and $r_F(z)=s_F(z)=z^2$. Then $E$ and $F$ are not isomorphic as topological graphs: $F^1$ is connected and $E^1$ is not. We show that $X(E)$ is isomorphic to $X(F)$. For every $f\in C(E^1)$ and $z\in \mathbb{T}$,
\[\inner{f}{f}(z)=\sum_{s_E(w,j)=z} |f(w,j)|^2=|f(z,1)|^2+|f(z,2)|^2,\]
therefore $\inner{f}{f}\in C(E^0)$. Hence \eqref{eq:defX(E)graphcorrespondence} gives $X(E)=C(E^1)$ and, analogously, $X(F)=~C(\mathbb{T})$.
Define $\psi:C(E^1)\to C(\mathbb{T},\mathbb{C}^2)$ by
\[\psi(h)(z)=\left(h(z,1),h(z,2)\right)\quad \text{for all }h\in C(E^1), z\in \mathbb{T}.\]
This map is an isomorphism of vector spaces, with inverse given by
\[\phi(g)(z,j)=\pi_j(g(z))\quad g\in C(\mathbb{T},\mathbb{C}^2), (z,j)\in \mathbb{T}\times \{0,1\}.\]
The map $\psi$ induces a Hilbert $C(\mathbb{T})$-bimodule structure on $C(\mathbb{T},\mathbb{C}^2)$ as follows. Writing $\inner{\cdot}{\cdot}_{\mathbb{C}}$ for the inner product in $\mathbb{C}^2$ (linear in the second component), for $g_1,g_2\in C(\mathbb{T},\mathbb{C}^2)$ and $f\in C(\mathbb{T})$,
\begin{align*}
    (f\cdot g_1)(z)&:=\psi(f\cdot \phi(g_1))(z)=f(z)g_1(z),\\
    (g_1\cdot f)(z)&:=\psi(\phi(g_1)\cdot f)=g_1(z)f(z), \quad \text{and}\\
    \inner{g_1}{g_2}(z)&:=\inner{\psi^{-1}(g_1)}{\psi^{-1}(g_2)}(z)=\inner{g_1(z)}{g_2(z)}_{\mathbb{C}}.
\end{align*}
Hence $\psi$ is an isomorphism $C(E^1)\to C(\mathbb{T},\mathbb{C}^2)$.

For $t\in [0,2\pi]$, define

\[\mathcal{U}_t=\begin{pmatrix}
e^{it/2}\cos(t/4) & -e^{it/2}\sin(t/4)\\
\sin (t/4) & \cos (t/4)
\end{pmatrix}\in \mathcal{U}(\mathbb{C}^2).\]
For $f\in C(\mathbb{T})$ define $\widetilde{\rho}(f):[0,2\pi]\to \mathbb{C}^2$ by
\[\widetilde{\rho}(f)(t)=\mathcal{U}_{t}\begin{pmatrix}
f(e^{it/2}) \\
f(-e^{it/2})
\end{pmatrix}.\]
Then $\widetilde{\rho}$ is continuous, and 
\[\widetilde{\rho}(f)(0)=\begin{pmatrix}
f(1) \\
f(-1)
\end{pmatrix}=\begin{pmatrix}
0& 1 \\
1& 0
\end{pmatrix}\begin{pmatrix}
f(-1) \\
f(1)
\end{pmatrix}=\widetilde{\rho}(f)(2\pi).\]
Hence $\widetilde{\rho}$ determines a continuous function $\rho(f)\in C(\mathbb{T},\mathbb{C}^2)$ by $\rho(f)(e^{it})=\widetilde{\rho}(f)(t)$. Identifying $X(F)$ with $C(\mathbb{T})$ as above, we obtain maps,
\[\begin{tikzcd}
X(F) \arrow[r, "\rho"] & {C(\mathbb{T},\mathbb{C}^2)} \arrow[r, "\phi"] & X(E).
\end{tikzcd}\]
We show that $\rho$ is an isomorphism of Hilbert bimodules. We start by proving that it is isometric. For $f_1,f_2\in C(\mathbb{T})$, we calculate
\begin{align*}
    \inner{\rho (f_1)}{\rho (f_2)}(e^{it}) &=\inner{\rho(f_1)(e^{it})}{\rho(f_2)(e^{it})}_{\mathbb{C}}\\
    &=\inner{\mathcal{U}_t\begin{pmatrix}
f_1(e^{it/2}) \\
f_1(-e^{it/2})
\end{pmatrix}}{\mathcal{U}_t\begin{pmatrix}
f_2(e^{it/2}) \\
f_2(-e^{it/2})
\end{pmatrix}}_{\mathbb{C}}\\
&=\inner{\begin{pmatrix}
f_1(e^{it/2}) \\
f_1(-e^{it/2})
\end{pmatrix}}{\begin{pmatrix}
f_2(e^{it/2}) \\
f_2(-e^{it/2})
\end{pmatrix}}_{\mathbb{C}}\\
&=\inner{f_1}{f_2}(e^{it}).
\end{align*}
To see that $\rho$ preserves the bimodule structure, fix $f\in X(F)$ and $a\in C(\mathbb{T})$. Then 
\begin{align*}
 \rho(f\cdot a)(e^{it})&=\mathcal{U}_{t}\begin{pmatrix}
f\cdot a(e^{it}) \\
f\cdot a(-e^{it})
\end{pmatrix}= \mathcal{U}_{t}\begin{pmatrix}
f(e^{it})a(s_F(e^{i t})) \\
f(-e^{it}) a(s_F(-e^{it}))
\end{pmatrix}\\
&=\mathcal{U}_{t}\begin{pmatrix}
f(e^{i\pi t}) \\
f(-e^{i\pi t})
\end{pmatrix}a(e^{i t})=\rho(f)(e^{it})a(e^{it})=(\rho(f)\cdot a) (e^{it}).
\end{align*}
Since the left and right actions on each of $X(E)$ and $X(F)$ coincide, $\rho(a\cdot f)=a\cdot\rho(f)$ as well. It remains to prove that $\rho$ is surjective. For that, we use the Stone-Weierstrass theorem for Banach bundles \cite[Corollary 4.3]{gierz2006bundles}. We must show that $\rho(C(\mathbb{T}))\subset C(\mathbb{T},\mathbb{C}^2)$, when viewed as sections on the canonical bundle associated with $C(\mathbb{T},\mathbb{C}^2)$, is fiberwise dense in that bundle. First, we define some sets. For $z\in \mathbb{T}$, let
\begin{align*}
    I_z&\coloneqq\{f\in C(\mathbb{T}): f(z)=0\},\\
    C(\mathbb{T},\mathbb{C}^2)_z&\coloneqq \{h\in C(\mathbb{T},\mathbb{C}^2): \inner{h}{h}\in I_z\}=\{h\in C(\mathbb{T},\mathbb{C}^2): h(z)=0\}.
\end{align*}
For $h\in C(\mathbb{T},\mathbb{C}^2)$, the associated section of the bundle $\bigsqcup_{z\in \mathbb{T}}C(\mathbb{T},\mathbb{C}^2)/C(\mathbb{T},\mathbb{C}^2)_z$ is given by $\hat{h}:z\mapsto h+C(\mathbb{T},\mathbb{C}^2)_z$. The map $h\mapsto \hat{h}$ is an isometric isomorphism. For $z\in\mathbb{T}$ write $\operatorname{ev}_z:C(\mathbb{T},\mathbb{C}^2)\to C(\mathbb{T},\mathbb{C}^2)/C(\mathbb{T},\mathbb{C}^2)_z$ for the ``evaluation'' map
\[\operatorname{ev}_z(h)=h+C(\mathbb{T},\mathbb{C}^2)_z\quad h\in C(\mathbb{T},\mathbb{C}^2).\]
We must show that $\operatorname{ev}_z(\rho(C(\mathbb{T})))$ is dense in $C(\mathbb{T},\mathbb{C}^2)/C(\mathbb{T},\mathbb{C}^2)_z$ for each $z\in \mathcal{T}$. For this, fix $z=e^{it}\in \mathbb{T}$ and $h+C(\mathbb{T},\mathbb{C}^2)_z\in C(\mathbb{T},\mathbb{C}^2)/C(\mathbb{T},\mathbb{C}^2)_z$. We show that there exists $f\in C(\mathbb{T})$ such that $h-\rho(f)\in C(\mathbb{T},\mathbb{C}^2)_z$ or, in other words, $h(e^{it})=\rho(f)(e^{it})$. First, we solve the equation,
\begin{equation}\label{eq.exampleequation}
   \mathcal{U}_{t}\begin{pmatrix}
x \\
y
\end{pmatrix}=h(e^{it}).
\end{equation}
Since $\mathcal{U}_t$ is unitary, there is a unique solution $(x_0,y_0)\in \mathbb{C}^2$. Let $A=\{-e^{it/2},e^{it/2}\}$ and define $\widetilde{f}:A\to \mathbb{C}$ by $\widetilde{f}(e^{it/2})=x_0$ and $\widetilde{f}(-e^{it/2})=y_0$. By Tietze's extension Theorem there exists $f\in C(\mathbb{T})$ such that  $f(e^{it/2})=x_0$ and $f(-e^{it/2})=y_0$. This implies that $\rho(f)(e^{it})=h$. Therefore, by Corollary 4.3 of \cite{gierz2006bundles} $\rho(C(\mathbb{T}))$ is dense in $C(\mathbb{T},\mathbb{C}^2)$ and since the former is closed, they are equal.

\section{A cohomological obstruction}\label{sec.epilogue}

We finish with a discussion of the relationship between isomorphism of topological graphs, isomorphism of the associated $C^*$-algebraic triples, isomorphism of their graph bimodules, and local conjugacy of topological graphs. To begin this discussion, we make an observation about the structure of graph bimodules in terms of the description of vector bundles using local trivialisations and transition functions. This is essentially a rephrasing of Kaliszewski, Patani and Quigg's characterisation of graph modules as those admitting a continuous choice of basis \cite{KaliPataQuigg}.

\begin{proposition}\label{prp:permutation transitions}
 Let $E$ be a compact topological graph. Then there exists a local trivialisation of the canonical Hilbert bundle $\mathcal{E}$ associated to the Hilbert $C(E^0)$-module $X(E)$ whose transition functions take values in the permutation matrices.
\end{proposition}
\begin{proof}
By Lemma~\ref{lemma:s-sectionslocalconstant} there exists a finite open cover $\{U_i\}_{i\in F}$,  of the vertex space $E^0$, such that, for each $i\in F$ there exists $k(i) \in \mathbb{N}$ and $s$-sections $Z_j^{(i)}, j \le k(i)$ such that the source map restricts to a homeomorphism $Z_j^{(i)} \cong U_i$ for each $j$, and $E^1 U_i = \bigsqcup_{j=1}^{k(i)}Z_j^{(i)}$.

Fix $i\in F$ and for each $j \le k(i)$, define $\phi_i^j:Z_j^{(i)}\to U_i\times \{1,\ldots, k(i)\}$ by $\phi_i^j(e)=(s|_{Z_j^{(i)}}(e),j)$. Then each $\phi_i^j$ is a homeomorphism of $Z_j^{(i)}$ onto $U_i\times \{j\}$. Hence, by the pasting lemma, $\phi_i\coloneqq \bigsqcup_{j=1}^{k(i)}\phi_i^j : E^1 U_i \to U_i\times \{1,\ldots, k(i)\}$ is a homeomorphism.  

If $U_i\cap U_j\neq \emptyset$, then we may consider the composition $\phi_j\circ \phi_i^{-1}$ with domain $(U_i\cap U_j)\times \{1,\ldots,k(i)\}$. Let $(x,\ell)\in (U_i\cap U_j) \times \{1,\ldots,k(i)\}$. Then $\phi_i^{-1}(x,\ell) = (s|_{Z_\ell^{(i)}})^{-1}(x)$, and there exists a $m\in \{1,\ldots, k(j)\}$ such that $(s|_{Z_\ell^{(i)}})^{-1}(x)\in Z_m^{(j)}$, since otherwise $(s|_{Z_\ell^{(i)}})^{-1}(x)\notin E^1 U_j$,  contradicting $x\in U_j$. This $m$ is unique because the $Z_n^{(j)}$, $n\in \{1,\ldots, k(j)\}$ are disjoint. Hence
\[
    \phi_j\circ \phi_i^{-1}(x,\ell) = \phi_j((s|_{Z_\ell^{(i)}})^{-1}(x)) = (x,m).
\]
So there is a function $\sigma_x^{j,i} : \{1, \dots, k(i)\} \to \{1, \dots, k(j)\}$ such that $(s|_{Z_\ell^{(i)}})^{-1}(x)\in Z^{(j)}_{\sigma_x^{j,i}(\ell)}$ for all $\ell\in \{1,\ldots, k(i)\}$.

We claim that $\sigma_x^{j,i}$ is a bijection. Indeed, if $\sigma_x^{j,i}(\ell)=\sigma_x^{j,i}(\ell')$, then $\phi_j\circ\phi_i^{-1}(x,\ell) = (x,\sigma_x^{j,i}(\ell)) = (x,\sigma_x^{j,i}(\ell')) = \phi_j\circ\phi_i^{-1}(x,\ell')$. Since $\phi_j$ and $\phi_i$ are injective, $\ell=\ell'$, thus $\sigma_x^{j,i}$ is injective.

We now claim that $\sigma^{j,i}_x$ is surjective. For this, fix $m\in \{1,\ldots,k(j)\}$. Then $\phi_j^{-1}(x,m) = (s|_{Z_m^{(j)}})^{-1}(x) \in  E^1(U_i\cap U_j) \subseteq E^1 U_i$. Thus, there is a unique $m'\in \{1,\ldots, k(i)\}$ such that $\phi_j^{-1}(x,m)\in Z_{m'}^{(i)}$, so $\phi_j^{-1}(x,m) = (s|_{Z_{m'}^{(i)}})^{-1}(x)$. Hence
\[
\phi_j\circ \phi_i^{-1}(x,m')
    =\phi_j((s|_{Z_{m'}^{(i)}})^{-1}(x))
    =\phi_j\big(\phi_j^{-1}(x,m)\big)
    =(x,m).
\]
This implies that $\sigma^{j,i}_x(m')=m$ and $\sigma^{j,i}_x$ is a bijection. In particular, $k(i)=k(j)$.

Recall the canonical Hilbert bundle $\mathcal{E}$ associated to the bimodule $X(E)$ discussed in the preliminaries. For each $i$ the map $\phi_i$ induces a right-Hilbert $C_0(U_i)$-module isomorphism $\phi_i^* : C_0(U_i\times \{1,\ldots, k(i)\}) \cong C_0(U_i, \mathbb{C}^{k(i)}) \to X(E) \cdot C_0(U_i)$ by $\phi_i^*(\xi) = \xi \circ \phi_i$. This induces a vector-bundle isomorphism $\psi_i : U_i \times \mathbb{C}^{k(i)} \cong \mathcal{E}|_{U_i}$ satisfying $\psi_i(x, e_m) = \phi_i^{-1}(x, m)$ for all $x \in U_i$ and $m \le k(i)$.

The maps $\psi_i, i \in F$ are a local trivialisation of $\mathcal{E}$, and for $x \in U_i \cap U_j$ and $\ell \le k(i)$, the transition function $\psi_j^{-1} \circ \psi_i$ satisfies $\psi_j^{-1} \circ \psi_i(x, e_\ell) = \psi_j^{-1}(\phi_i^{-1}(x,\ell)) = \psi_j^{-1}(\phi_j^{-1}(\sigma_x^{j,i}(\ell))) = (x, e_{\sigma_x^{j,i}(\ell)})$. That is, the matrix implementing $\psi_j^{-1} \circ \psi_i$ in the fibre over $x$ is precisely the permutation matrix corresponding to $\sigma_x^{j,i}$. 
\end{proof}

\begin{corollary}\label{cor:cts basis cohomology}
Let $K$ be a second-countable compact Hausdorff space, and let $X$ be a right-Hilbert $C(K)$-module. The following are equivalent.
\begin{enumerate}
    \item $X$ is isomorphic, as a right-Hilbert module, to the graph module of a compact topological graph $E$ with $E^0 \cong K$.
    \item $X$ admits a continuous choice of finite orthonormal bases in the sense of Kaliszewksi--Quigg--Patani \cite{KaliPataQuigg}.
    \item The canonical vector bundle $\mathcal{E}$ over $K$ associated to $X$ is finite rank and admits a local trivialisation whose transition functions take values in the permutation matrices.
\end{enumerate}
\end{corollary}
\begin{proof}
That (1)~and~(2) are equivalent follows from \cite[Theorem~6.4]{KaliPataQuigg}. Proposition~\ref{prp:permutation transitions} gives \mbox{(1)$\implies$(3)}. So it suffices to show that if $\mathcal{E}$ has a local trivialisation whose transition functions take values in permutation matrices, then it admits a global choice of orthonormal bases. For this, fix such a local trivialisation, say $\{U_i\}_{i \in F}$ is an open cover of $K$ and for each $i$, we have a bundle isomorphism $\psi_i : U_i \times \mathbb{C}^{k(i)} \to \mathcal{E}|_{U_i}$ so that the transition functions $\psi_i^{-1} \circ \psi_j$ are permutation-matrix valued. Fix $x \in K$. If $i,j \in F$ satisfy $x \in U_i \cap U_j$ then since $\psi_i^{-1} \circ \psi_j$ takes values in permutation matrices, there is a permutation $\sigma$ such that $\psi_i^{-1} \circ \psi_j(x, e_\ell) = (x, e_{\sigma(\ell)})$ for all $\ell\le k(i)$. Hence $\{\psi_i(x, e_1), \dots, \psi_i(x, e_{k(i)})\} = \{\psi_j(x, e_{\sigma(1)}), \dots, \psi_j(x, e_{\sigma({k(i)})})\} = \{\psi_j(x, e_1), \dots, \psi_j(x, e_{k(i)})\}$. So there is a well-defined set-valued map $\psi : K \to \mathcal{P}(\mathcal{E})$ such that, for each $i$, the restriction $\psi|_{U_i}$ is the map $x \mapsto \{\psi_i(x, e_1), \dots, \psi_i(x, e_{k(i)})\}$. The set $B := \bigcup_{x \in K} \psi(x)$ is a subset of $\mathcal{E}$. Each $B \cap \mathcal{E}_x = \psi(x)$ is an orthonormal basis for $\mathcal{E}_x$. For each $i$, the restriction $\pi_i$ of the bundle map to $B \cap \mathcal{E}|_{U_i}$ satisfies $\pi_i(\psi_i(x, e_\ell)) = x$ and for each $x,\ell$ the set $\psi_i(U_i \times \{e_\ell\}) \cong U_i \times \{e_\ell\}$ is an open neighbourhood of $(x, e_\ell)$ on which $\pi_i$ restricts to a homeomorphism. So the bundle map restricts to a local homeomorphism $B \to K$. That is $B$ is a continuous choice of basis.
\end{proof}

\begin{remark}
Our results give the following string of implications for a pair $E$ and $F$ of compact topological graphs:
\begin{align*}
E \cong F \stackrel{(1)}{\Longrightarrow}{}& {(\mathcal{T}C^*(E), \gamma^E, M_E) \cong (\mathcal{T}C^*(F), \gamma^F, M_F)} \\
   \Longleftrightarrow{}& X(E) \cong X(F) \\
   \stackrel{(2)}{\Longrightarrow}{}& E \cong_{\operatorname{loc}} F.
\end{align*}

Section~\ref{section.example} details a pair $E, F$ of topological graphs that have isomorphic modules (and so, in particular, are also locally conjugate) but are not isomorphic as topological graphs. Thus implication~(1) does not admit a converse.

Though we do not have a counterexample, it also seems unlikely that implication~(2) admits a converse due to the following cohomological considerations. 

Corollary~\ref{cor:cts basis cohomology} shows that graph modules of compact topological graphs, regarded purely as right modules, can be characterised as the finite-rank vector bundles that admit a local trivialisation whose transition functions take values in the permutation matrices. Any (right) graph module can be made into a graph bimodule by making the left action identical to the right action; this amounts to making all of the edges in the graph loops. 

Hence, the existence of a nontrivial vector bundle over a compact space whose transition functions are permutation matrices would provide a counterexample to the converse of implication~(2): given such a bundle $B$, say of rank $k$, Corollary~\ref{cor:cts basis cohomology} yields a topological graph $E$ whose graph module is the module of sections of $B$. By construction, the range and source maps on $E$ coincide, and each vertex of $E$ is the base of $k$ loop-edges. Hence $E$ is locally conjugate to the topological graph $F$ with $F^0 = E^0$, and $F^1 = E^0 \times \{1, \dots, k\}$ as a topological space, and with range and source maps given by $(v, i) \mapsto v$. Since the graph module of $F$ is the module of sections of a trivial bundle, it follows that $E$ and $F$ are locally conjugate topological graphs with non-isomorphic graph modules.

The larger question is how to characterise exactly what additional cohomological data should be paired with a local conjugacy of graphs to determine an isomorphism of graph bimodules. The question seems complicated since one cannot directly apply the standard classification theory for vector bundles---a more refined cohomology is required that takes into account the left action of $C(E^0)$ on $X(E)$, likely via the functions $\alpha_i$ appearing in Lemma~\ref{lemma.principallemmatoprovetheorem}(3).
\end{remark}

The discussion above is consistent with the results of \cite{DavRoy}. As mentioned in Section~\ref{section.local}, their theorem says that if the covering dimension of $E^0$ is at most~1, then local conjugacy of topological graphs implies isomorphism of their Hilbert modules. Our proposed strategy for constructing locally conjugate graphs with non-isomorphic Hilbert bimodules is outside the scope of that theorem: if the covering dimension of $E^0$ is at most~1, then there are no nontrival vector bundles over $E^1$---with transition functions taking values in permutation matrices or otherwise.

\end{document}